\documentclass[a4paper]{article}
\usepackage{eurosym}
\usepackage[pages=all, color=black, position={current page.south}, placement=bottom, scale=1, opacity=1, vshift=5mm]{background}
\usepackage[margin=1in]{geometry}
\usepackage{amsmath}
\usepackage{amsthm}
\usepackage{amssymb}
\usepackage{amssymb}
\usepackage{amsfonts}
\usepackage{amsmath}
\usepackage{amssymb}
\usepackage{amsfonts}

\usepackage{graphicx}
\usepackage{subcaption}
\usepackage{float}

\usepackage{graphicx}
\usepackage{caption}
\usepackage[T1]{fontenc}
\usepackage[utf8]{inputenc}
\usepackage{hyperref}
\usepackage[sort&compress,numbers,square]{natbib}
\usepackage{graphicx, color}
\usepackage{algorithm, algpseudocode}
\usepackage{mathrsfs}
\usepackage{lipsum}
\usepackage{rotating} 
\usepackage{float}% include in preamble
\usepackage{changepage}
\usepackage{pdflscape}
\usepackage{tikz}
\usetikzlibrary{positioning,arrows.meta,shapes.geometric}

\SetBgContents{
}      
\hypersetup{
unicode,
colorlinks,
breaklinks,
urlcolor=cyan, 
linkcolor=blue, 
pdfauthor={Author One, Author Two, Author Three},
pdftitle={A simple article template},
pdfsubject={A simple article template},
pdfkeywords={article, template, simple},
pdfproducer={LaTeX},
pdfcreator={pdflatex}
}
\theoremstyle{plain}
\newtheorem{theorem}{Theorem}

\theoremstyle{definition}

\begin{document}

\title{Direct Finite-Time Contraction (Step-Log) Profiling--Driven Optimization of Parallel Schemes for Nonlinear Problems on Multicore Architectures}
\author{Mudassir Shams$^1$$^,$$^2$$^,$$^*$, Andrei Velichko$^3$, Bruno Carpentieri$^2$ }
\date{$^1$Department of Mathematics, Faculty of Arts and Science, Balikesir
University, Balikesir 10145, Turkey\\
$^2$Faculty of Engineering, Free University of Bozen-Bolzano, 39100,
Bolzano, Italy\\
%$^3$Department of Mathematics and Statistics, Riphah International
%University I-14, Islamabad 44000, Pakistan\\
$^3$Institute of Physics and Technology, Petrozavodsk State University, 185910 Petrozavodsk, Russia \\
\texttt{mudassir.shams@balikesir.edu.tr;velichko@petrsu.ru}\\
\texttt{bruno.carpentieri@unibz.it}\\
\texttt{$^*$Correspondence:Mudassir Shams: Email:mudassir.shams@balikesir.edu.tr}\\
[2ex]\today
}
\maketitle

\begin{abstract}
Efficient computation of all distinct solutions of nonlinear problems is essential in many scientific and engineering applications. Although high-order parallel iterative schemes offer fast convergence, their practical performance is often limited by sensitivity to internal parameters and the lack of reproducible tuning procedures. Classical parameter selection tools based on analytical conditions and dynamical-system diagnostics can be problem-dependent and computationally demanding, which motivates lightweight data-driven alternatives.

In this study, we propose a parameterized single-step bi-parametric parallel Weierstrass-type scheme with third-order convergence together with a training-free tuning framework based on \emph{Direct finite-time contraction (step-log) profiling}. The approach extracts Lyapunov-like finite-time contraction information directly from solver trajectories via step norms and step-log ratios, aggregates the resulting profiles over micro-launch ensembles, and ranks parameter candidates using two compact scores: the stability minimum $S_{\min}$ and the stability moment $S_{mom}$. Numerical results demonstrate consistent improvements in convergence rate, stability, and robustness across diverse nonlinear test problems, establishing the proposed profiling-based strategy as an efficient and reproducible alternative to classical parameter tuning methods.
\end{abstract}

\section{Introduction}

Nonlinear equations arise ubiquitously in mathematical models of physical, engineering, and biological systems. Problems in fluid dynamics, chemical kinetics, electrical circuits, control theory, biomedical engineering, and optimization often lead to nonlinear algebraic equations whose analytical solutions are unavailable~\cite{1,2}. As a result, iterative numerical methods remain the primary computational tools for approximating their solutions. Although the theory of nonlinear equations is well established, the practical efficiency of numerical solvers~\cite{3} strongly depends on the design of the iterative scheme and, crucially, on the appropriate choice of method parameters~\cite{4}. Modern iterative methods for nonlinear problems are rarely parameter-free. Parameters naturally appear when higher-order corrections, weight functions, or acceleration strategies are introduced to improve convergence speed and stability. While these parameters provide additional flexibility, they also introduce a significant challenge: the convergence order, convergence rate, numerical stability, and robustness of the scheme may vary dramatically with parameter values. In many applications, particularly those involving high-degree polynomials or nonlinear systems with cluster solutions, it is necessary to compute all distinct roots rather than a single solution. This requirement has motivated the development of parallel iterative schemes~\cite{5}, which approximate all roots in a unified computational process utilizing parallel architecture on a multi-processor on a computing machine~\cite{6}. Among these, Weierstrass-type methods and their variants are widely regarded as fundamental parallel root-finding techniques.

In parameterized parallel schemes, this dependence is even more pronounced due to nonlinear interactions among the approximations of different roots. A parameter choice that performs well for one class of problems may lead to slow convergence, divergence, or chaotic behavior for another. Consequently, parameter tuning is not a secondary issue but a central component in the practical effectiveness of efficient iterative schemes for nonlinear problems.

Traditionally, parameter selection in iterative methods has relied on analytical convergence analysis, heuristic rules~\cite{7}, and extensive numerical experimentation. In the context of parallel schemes, tools from real and complex dynamical~\cite{8} systems have played a prominent role. These include parametric planes~\cite{9}, bifurcation diagrams~\cite{10}, critical point analysis, and basin-of-attraction visualizations.

Although such tools provide valuable qualitative insight into the global dynamics of iterative methods, they exhibit several limitations.
\begin{itemize}
    \renewcommand{\labelitemi}{--}
	\item First, they are computationally intensive and often restricted to low-dimensional parameter spaces.
    \item Second, their interpretation may depend on visual inspection and subjective judgment. 
    \item Third, they are not easily adaptable to automated or large-scale parameter optimization.
\end{itemize} 
As a result, the selection of optimal parameters often remains problem-dependent and lacks systematic reproducibility~\cite{11}. Parallel root-finding schemes are particularly sensitive to parameter choices because the update of each root approximation depends on the collective behavior of all others. Small parameter variations can significantly alter the balance between convergence acceleration and numerical stability. From a computational perspective, optimal parameter tuning directly translates into fewer iterations, reduced computational cost, and improved robustness when computing all distinct roots. Therefore, the development of reliable parameter optimization strategies is essential for fully exploiting the potential of high-order parallel iterative schemes, especially in large-scale or real-time applications.

% ===== Replace the original paragraph that starts with:
% "Recent advances in machine learning provide powerful data-driven alternatives ..."
% with the following text =====

Recent advances in machine learning provide powerful data-driven alternatives for parameter tuning in numerical algorithms~\cite{12}. In particular, learning-based estimators can infer stability indicators directly from trajectory data and thereby complement classical analytical conditions and visual dynamical diagnostics. In our preliminary investigations, we used a kNN-based learning estimator for largest Lyapunov exponents (LLE)~\cite{13} as an exploratory diagnostic tool and observed that \emph{finite-time Lyapunov-style profiles} computed from solver trajectories are highly sensitive to parameter-induced stability changes in parallel root-finding schemes. In practice, however, such learning-based pipelines typically require repeated neighborhood searches, embedding/model choices, and additional hyperparameter tuning, which increases the computational and implementation overhead when a dense parameter grid must be scanned.

Motivated by these observations, in this work we introduce a lightweight Lyapunov-like alternative termed \emph{Direct finite-time contraction (step-log) profiling}. The method avoids training and manifold reconstruction and instead extracts contraction information directly from iteration histories. Specifically, we compute step norms and their step-log ratios to form finite-time contraction profiles, which act as a practical proxy for Lyapunov-type stability behavior. Two compact profile-based performance metrics, namely the stability minimum $S_{\min}$ and the stability moment $S_{mom}$, are then used to quantify convergence quality across a diverse set of nonlinear test problems and to identify parameter values that consistently yield fast and stable convergence.

% ===== Replace the "main contributions" itemize environment with the following =====

The main contributions of this study are summarized as follows:
\begin{itemize}
\item Development of a parameterized parallel Weierstrass-type scheme accelerated by a single-step correction mechanism, increasing the convergence order from two to three.
\item A systematic investigation of the influence of the real-valued parameters on convergence order, convergence rate, and stability of the parallel iteration.
\item Introduction of \emph{Direct finite-time contraction (step-log) profiling} as a training-free, Lyapunov-like diagnostic for detecting parameter-induced stability changes from solver trajectories.
\item Definition and application of two profile-based scores, the minimum-based metric $S_{\min}$ and the moment-based metric $S_{mom}$, enabling reproducible parameter selection over scanned parameter grids.
\item Extensive numerical evidence across diverse nonlinear and application-driven test problems demonstrating consistent improvements in convergence speed, robustness, and stability for simultaneous computation of distinct roots.
\end{itemize}

% --- Replace the ending of the Introduction with the following ---

\noindent\textbf{Novelty and positioning.}
The main contribution of this study is a \emph{general, quantitative framework}
for \emph{finite-time parameter optimization} in high-order parallel iterative
methods for nonlinear equations.
Rather than relying on qualitative dynamical visualizations
(e.g., basins of attraction, bifurcation-style parameter inspection),
we propose a \emph{direct finite-time contraction (step-log) profiling} approach
that extracts Lyapunov-like contraction information \emph{directly} from solver
trajectories and summarizes it through compact, reproducible scores.
The framework is demonstrated on a bi-parametric, third-order Weierstrass-type
parallel scheme (SAB$^{[3]}$), where the profile-based metrics
$S_{\min}$ and $S_{mom}$ provide an automated and implementation-light rule for
selecting robust parameter regions over scanned grids.
Learning-based kNN/LLE estimators are mentioned only as preliminary motivation:
in contrast, the proposed profiling pipeline is \emph{training-free} and avoids
model-selection overhead, while remaining sensitive to parameter-induced
stability changes in the transient iteration regime.

\medskip
\noindent\textbf{Paper organization.}
The remainder of the paper is organized as follows.
Section~2 presents the proposed parallel iterative scheme and its theoretical
convergence properties.
Section~3 introduces the direct finite-time contraction (step-log) profiling pipeline,
defines the aggregated profiles and the metrics $S_{\min}$ and $S_{mom}$, and describes
the parameter-plane exploration and selection strategy.
Section~4 reports numerical experiments and comparative studies across several
nonlinear test problems.
Finally, Section~5 concludes the paper and outlines limitations and future research directions.

%\begin{itemize}
%    \renewcommand{\labelitemi}{--}
%	\item Development
%\end{itemize}

%\section{Development, Methodology, and Implementation of Lyapunov Profiling}

\section{Proposed Parallel Schemes and Convergence Methodology}

This section presents the proposed inverse parallel iterative schemes and outlines the adopted methodological framework. We briefly motivate the use of parallel root-finding strategies and recall classical schemes that underpin the proposed approach see e.g., ~\cite{15,16,17}.

The mathematical formulation of the methods is provided together with a Lyapunov-based convergence analysis, a direct finite-time contraction (step-log) profiling procedure, and a summary of the experimental design and numerical evaluation.

Following the impossibility theorems~\cite{18}, parallel methods have become essential for approximating all roots of nonlinear equations,
\begin{equation}
f(x)=0,\label{1Q}
\end{equation} 
particularly polynomials, due to their global convergence, numerical stability, and robustness to initial guesses~\cite{19}. Their simultaneous computation of distinct and multiple roots, combined with an inherently parallel structure suitable for multicore, GPU, and distributed architectures, ensures scalable, efficient, and reliable performance in large-scale scientific and engineering applications.

Among the classical simultaneous root-finding schemes, the well-known
Weierstrass--Durand--Kerner (WDK) method~\cite{20} is widely used due to its
simplicity and local quadratic convergence. It is defined as
\begin{equation}
x_{i}^{[h+1]}=x_{i}^{[h]}-\frac{f\!\left(x_{i}^{[h]}\right)}
{\displaystyle\prod_{\substack{ j=1 \\ j\neq i}}^{n}\left(x_{i}^{[h]}-x_{j}^{[h]}\right)}.
\label{1c}
\end{equation}

To enhance the convergence order, Nourein~\cite{20a} proposed a third-order modification,
denoted by PNS$^{[3]}$, which replaces the product in the denominator with a
weighted summation in (\ref{1c}):
\begin{equation}
x_{i}^{[h+1]}=x_{i}^{[h]}-\frac{f\!\left(x_{i}^{[h]}\right)}
{\displaystyle\sum_{\substack{ j=1 \\ j\neq i}}^{n}\left(
\frac{\mathcal{P}\!\left(x_{j}^{[h]}\right)}{x_{j}^{[h]}-x_{i}^{[h]}}
\right)},
\end{equation}
where
\[
\mathcal{P}\!\left(x_{i}^{[h]}\right)=
\frac{f\!\left(x_{i}^{[h]}\right)}
{\displaystyle\prod_{\substack{ j=1 \\ j\neq i}}^{n}\left(x_{i}^{[h]}-x_{j}^{[h]}\right)}.
\]

Petkovi\'{c} \emph{et al.}~\cite{21a} introduced another higher-order modification of (\ref{1c}),
referred to as PPS$^{[3]}$, achieving convergence of order $r>2$ through a
partial correction in the denominator:
\begin{equation}
x_{i}^{[h+1]}=x_{i}^{[h]}-
\frac{f\!\left(x_{i}^{[h]}\right)}
{\displaystyle
\prod_{j=1}^{i-1}\!\left(x_{i}^{[h]}-x_{j}^{[h]}-
\mathcal{P}\!\left(x_{i}^{[h]}\right)\right)
\prod_{j=i+1}^{n}\!\left(x_{i}^{[h]}-x_{j}^{[h]}\right)}.
\end{equation}

B\"{o}rsch--Supan~\cite{22a} proposed a third-order simultaneous scheme, denoted by
BSS$^{[3]}$, which incorporates correction terms into the denominator to
accelerate convergence:
\begin{equation}
x_{i}^{[h+1]}=x_{i}^{[h]}-
\frac{\mathcal{P}\!\left(x_{i}^{[h]}\right)}
{\displaystyle
1+\sum_{\substack{ j=1 \\ j\neq i}}^{n}
\left(\frac{\mathcal{P}\!\left(x_{j}^{[h]}\right)}
{x_{i}^{[h]}-x_{j}^{[h]}}\right)}.
\end{equation}

A further refinement was presented by Nourein in the fourth-order method
PNS$^{[4]}$~\cite{21}, which modifies the interaction term among approximations:
\begin{equation}
x_{i}^{[h+1]}=x_{i}^{[h]}-
\frac{\mathcal{P}\!\left(x_{i}^{[h]}\right)}
{\displaystyle
1+\sum_{\substack{ j=1 \\ j\neq i}}^{n}
\left(\frac{\mathcal{P}\!\left(x_{j}^{[h]}\right)}
{x_{i}^{[h]}-\mathcal{P}\!\left(x_{j}^{[h]}\right)-x_{j}^{[h]}}\right)}.
\end{equation}

In addition to these classical schemes, numerous multi-step and high-order simultaneous methods have been developed to improve convergence speed and numerical stability. Notable contributions include those of Petković et al.~\cite{23}, Proinov~\cite{24}, Mir et al.~\cite{25}, Nedzhibov~\cite{26}, and Cordero et al.~\cite{27}, among others (see, for example, \cite{28,29,30} and the references therein). These works provide the theoretical and algorithmic basis upon which the present inverse parallel schemes are constructed.

\bigskip

\noindent\textbf{Proposed Inverse Parallel Fractional Scheme.}
We begin with a parameterized single-root iterative scheme~\cite{31} for approximating a simple root of a nonlinear equation of (\ref{1Q}) as given 
\begin{equation}
x^{[h+1]}=x^{[h]}-\frac{f(x^{[h]})}{f^{\prime }(x^{[h]})}
\left(
\frac{1}{1+\frac{\alpha f(x^{[h]})}{1+\beta f(x^{[h]})}}
\right).\label{1a}
\end{equation}
By incorporating the classical Weierstrass correction, the above self-correcting formulation is extended to a parallel framework for the simultaneous approximation of all roots of a nonlinear equation:
\begin{equation}
x_{i}^{[h+1]}=x_{i}^{[h]}-
\frac{f(x_{i}^{[h]})}
{\displaystyle\prod_{\substack{ j=1 \\ j\neq i}}^{n}\left( x_{i}^{[h]}-z_{j}^{[h]}\right)}
,\label{1b}
\end{equation}
where the auxiliary predictor is defined as
\begin{equation*}
z_{j}^{[h]}=x_{j}^{[h]}-\frac{f(x_{j}^{[h]})}{f^{\prime }(x_{j}^{[h]})}\left(
\frac{1}{1+\frac{\alpha f(x_{j}^{[h]})}{1+\beta f(x_{j}^{[h]})}}
\right),
\end{equation*}
and $\alpha,\beta\in\mathbb{R}$ are free parameters controlling the acceleration behavior. The proposed scheme (\ref{1b}) is abbreviated as (SAB$^{[3]}$).

\subsection{Local Convergence Analysis}

We now analyze the local convergence properties of the proposed inverse parallel fractional scheme. The following theorem establishes its convergence order under standard smoothness assumptions.

\begin{theorem}
Let $\zeta _{1},\ldots ,\zeta _{\upsilon }$ be simple zeros of a nonlinear equation $f(x)=0$. If the initial approximations
$x_{1}^{[0]},\ldots ,x_{\upsilon }^{[0]}$ are sufficiently close to their corresponding exact roots, then the order SAB$^{[3]}$ method defined above converges with order three.
\end{theorem}

\begin{proof}
Let the iteration errors be defined as
\begin{equation*}
\varepsilon _{i}=x_{i}^{[h]}-\zeta _{i},
\qquad
\varepsilon _{i}^{\prime }=x_{i}^{[h+1]}-\zeta _{i}.
\end{equation*}
From the iterative formula, we obtain
\begin{equation}
x_{i}^{[h+1]}-\zeta _{i}
=
x_{i}^{[h]}-\zeta _{i}
-\dfrac{f(x_{i}^{[h]})}
{\displaystyle\prod_{\substack{ j=1 \\ j\neq i}}^{n}(x_{i}^{[h]}-z_{j}^{[h]})},
\label{25}
\end{equation}
where
\begin{equation*}
z_{j}^{[h]}=x_{j}^{[h]}-\frac{f(x_{j}^{[h]})}{f^{\prime }(x_{j}^{[h]})}\left(
\frac{1}{1+\frac{\alpha f(x_{j}^{[h]})}{1+\beta f(x_{j}^{[h]})}}
\right).
\end{equation*}

Rewriting \eqref{25} in terms of the error variables yields
\begin{equation}
x_{i}^{[h+1]}-\zeta _{i}
=
x_{i}^{[h]}-\zeta _{i}
-\varepsilon _{i}
\prod\limits_{\substack{ j=1 \\ j\neq i}}^{n}
\left(
\frac{x_{i}^{[h]}-\zeta _{j}}{x_{i}^{[h]}-z_{j}^{[h]}}
\right).
\label{26}
\end{equation}
Consequently,
\begin{equation}
\varepsilon _{i}^{\prime }
=
\varepsilon _{i}
\left(
1-
\prod\limits_{\substack{ j=1 \\ j\neq i}}^{n}
\left(
\frac{x_{i}^{[h]}-\zeta _{j}}{x_{i}^{[h]}-z_{j}^{[h]}}
\right)
\right).
\label{27}
\end{equation}

Using the expansion
\begin{equation}
\frac{x_{i}^{[h]}-\zeta _{j}}{x_{i}^{[h]}-z_{j}^{[h]}}
=
1+\frac{z_{j}^{[h]}-\zeta _{j}}{x_{i}^{[h]}-z_{j}^{[h]}},
\end{equation}
together with $z_{j}^{[h]}-\zeta _{j}=O(\varepsilon _{j}^{2})$ ~\cite{31} and
%\[
%\frac{1}{1+\frac{\alpha f(x_{j}^{[h]})}{1+\beta f(x_{j}^{[h]})}}
%=
%O(\varepsilon _{i}),
%\]
%we obtain
\begin{equation*}
\prod\limits_{\substack{ j=1 \\ j\neq i}}^{n}
\left(
\frac{x_{i}^{[h]}-\zeta _{j}}{x_{i}^{[h]}-z_{j}^{[h]}}
\right)
=
1+(n-1)O(\varepsilon _{j}^{2}).
\end{equation*}

Assuming comparable error magnitudes $|\varepsilon _{i}|=|\varepsilon _{j}|=|\varepsilon|$, it follows that
\begin{equation}
\varepsilon _{i}^{\prime }=O(\varepsilon ^{3}),
\label{30a}
\end{equation}
which confirms that the proposed method converges with order three.
\end{proof}

\subsection{Mathematical Description of the Method}

The performance of high-order parallel iterative schemes for nonlinear equations
is critically influenced by internal tuning parameters.
While classical convergence analysis characterizes local asymptotic behavior,
it often does not explain transient instabilities, slow early contraction,
or parameter-induced degradation observed in practical runs.
To address this gap, we employ a \emph{direct finite-time contraction (step-log) profiling}
strategy: from solver trajectories we compute step norms and their logarithmic ratios,
and aggregate their windowed averages into a compact contraction profile.
This profile serves as a \emph{Lyapunov-like} finite-time diagnostic of contractive versus
expansive dynamics, enabling reproducible comparison of parameter choices, validation of
stable operating regions, and systematic selection of tunable parameters beyond
trial-and-error procedures.

\section{Methodology and Parameter Optimization Framework}\label{sec:tuning_framework}

The goal of this section is to describe a practical and reproducible strategy for tuning the free parameters of the proposed parallel scheme SAB$^{[3]}$. The central premise is that parameter-induced stability changes are best detected at \emph{finite time}, during the transient phase of the iteration, rather than only through asymptotic local analysis. We therefore build an automated tuning pipeline based on \emph{Direct finite-time contraction (step-log) profiling}, which provides a lightweight Lyapunov-like proxy computed directly from solver trajectories.

\subsection{Why data-driven tuning is needed}

The convergence behavior of SAB$^{[3]}$ is sensitive to the parameters $(\alpha,\beta)$ and to the initial constellation of approximations. Small parameter variations may lead to slow contraction, oscillatory transients, or complete loss of stability for a subset of roots. Classical dynamical tools (parameter planes, bifurcation-style diagnostics, basin visualizations) can provide qualitative insight, but they are expensive to compute, hard to automate, and often rely on visual interpretation. Our objective is to replace subjective inspection by a compact set of \emph{trajectory-derived} stability indicators that can be evaluated routinely over a parameter grid.

\subsection{Direct finite-time contraction (step-log) profiling}

For a fixed parameter pair $(\alpha,\beta)$, we run an ensemble of $N_{\mathrm{ens}}$ independent solver launches (in the experiments, $N_{\mathrm{ens}}=50$) using randomized complex initial guesses (or small perturbations around a reference initialization). Let $\mathbf{x}^{[h]}\in\mathbb{C}^n$ denote the vector of all root approximations at iteration $h$. We record the step norm
\begin{equation}
s_h=\left\|\mathbf{x}^{[h+1]}-\mathbf{x}^{[h]}\right\|_2,
\end{equation}
and define the step-log contraction trace
\begin{equation}
g(h)=\log\!\left(\frac{s_{h+1}+\varepsilon}{s_h+\varepsilon}\right),
\end{equation}
where $\varepsilon>0$ is a small numerical safeguard. Negative values of $g(h)$ indicate contractive behavior from one iteration to the next.

To suppress noise and expose coherent transient structure, we apply a sliding window of length $W$ and form the finite-time contraction profile
\begin{equation}
\lambda_W(t)=\frac{1}{W}\sum_{h=t-W+1}^{t} g(h),
\qquad t=W,\ldots,K-1,
\end{equation}
where $K$ is the iteration budget. For each $(\alpha,\beta)$, profiles are computed for all ensemble runs and then aggregated by the ensemble mean (optionally accompanied by a dispersion band such as mean$\pm$std). The resulting aggregated profile acts as a Lyapunov-like finite-time stability proxy: sustained negative values correspond to reliable contraction, whereas weak/late negativity or rapid sign changes typically correlate with unstable or inefficient solver regimes.

\subsection{Profile-based scores $S_{\min}$ and $S_{mom}$}

To rank parameter pairs by compact scalar indicators, we extract two scores
from the \emph{aggregated} step-log contraction profile
$\bar{\lambda}^{(s)}_{W}(t_{\mathrm{end}})$ computed by the
direct finite-time contraction (step-log) profiling protocol.
Figure~\ref{fig:profile_example_scores} illustrates a typical profile together
with the quantities used in the score definitions.

Let
\[
y_{\min}=\min_{t_{\mathrm{end}}}\,\bar{\lambda}^{(s)}_{W}(t_{\mathrm{end}}),
\qquad
t_{\min}=\arg\min_{t_{\mathrm{end}}}\,\bar{\lambda}^{(s)}_{W}(t_{\mathrm{end}}).
\]
We define the minimum-based score
\begin{equation}
S_{\min}=\max\!\left\{0,\frac{-y_{\min}}{t_{\min}}\right\},
\end{equation}
which rewards \emph{early and deep} contraction episodes.

Depth alone, however, is not sufficient: some unstable regimes may exhibit a
sharp negative dip followed by a rapid loss of contraction.
To capture the \emph{temporal organization} and persistence of contraction,
we define the moment-based score. Using the negative part
$(u)_+=\max\{u,0\}$, let
\[
M_0=\sum_{t_{\mathrm{end}}}\big(-\bar{\lambda}^{(s)}_{W}(t_{\mathrm{end}})\big)_+,
\qquad
\bar{t}=
\frac{\sum_{t_{\mathrm{end}}} t_{\mathrm{end}}\,
\big(-\bar{\lambda}^{(s)}_{W}(t_{\mathrm{end}})\big)_+}{M_0}.
\]
Then
\begin{equation}
S_{mom}=\frac{M_0}{\bar{t}}.
\end{equation}
Large $S_{mom}$ indicates that contraction is not only strong but also
concentrated early and sustained over the iteration horizon, which empirically
correlates with faster and more robust convergence in simultaneous root
computation.

\begin{figure}[t]
  \centering
  \includegraphics[width=0.85\linewidth]{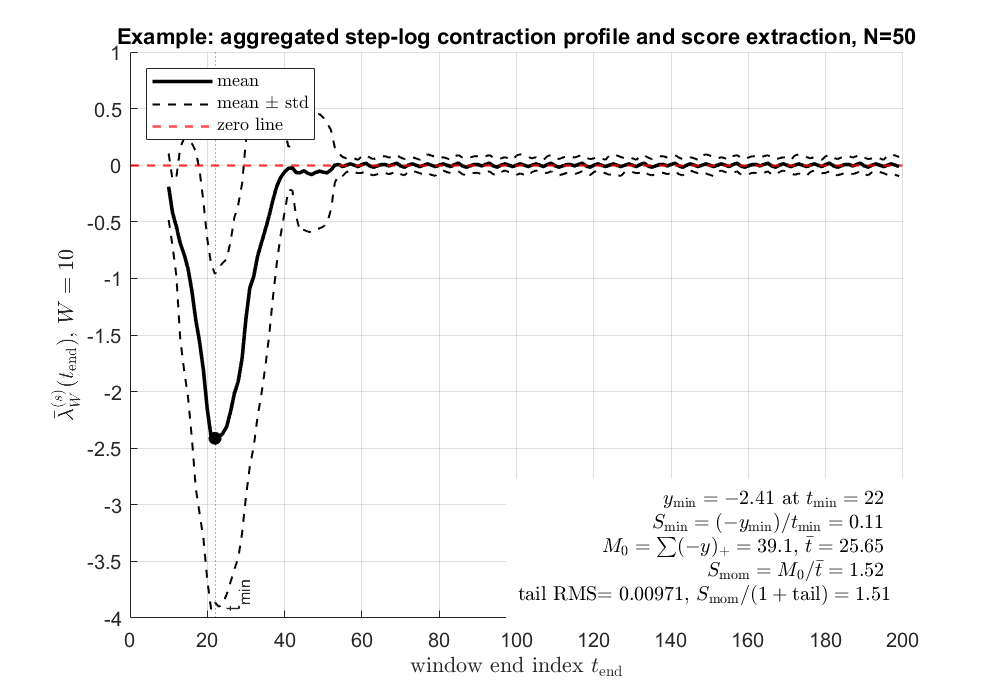}
  \caption{Example of an aggregated step-log contraction profile and the extraction
  of the profile-based scores. The solid curve shows the ensemble mean
  $\bar{\lambda}^{(s)}_{W}(t_{\mathrm{end}})$ and the dashed curves indicate
  mean$\pm$std across $N=50$ micro-launches (here $W=10$; parameters are shown in the plot title).}
  \label{fig:profile_example_scores}
\end{figure}

\subsection{Parameter scan and selection rule}

We evaluate $(S_{\min},S_{mom})$ over a rectangular grid in the $(\alpha,\beta)$ plane. For each grid point:
\begin{enumerate}
\item Run $N_{\mathrm{ens}}=50$ independent micro-launches of SAB$^{[3]}$ up to $K$ iterations (or until tolerance is met).
\item Compute step norms $s_h$ and step-log traces $g(h)$ for each run.
\item Form windowed profiles $\lambda_W(t)$ and aggregate them across the ensemble.
\item Extract $S_{\min}(\alpha,\beta)$ and $S_{mom}(\alpha,\beta)$ from the aggregated profile.
\end{enumerate}
Parameter selection is then performed by prioritizing robust sustained contraction, e.g.,
\[
(\alpha^\star,\beta^\star)=\arg\max_{(\alpha,\beta)} S_{mom}(\alpha,\beta),
\]
with $S_{\min}$ used as a tie-breaker when needed. Heatmaps of $S_{\min}$ and $S_{mom}$ provide an immediate, quantitative stability map of the parameter plane and enable reproducible reporting of high-performance regions. The final solver runs in Section~4 use the selected $(\alpha^\star,\beta^\star)$ and are compared against established parallel schemes under identical stopping criteria and initialization protocols.

\section{Numerical Experiments and Results}

This section reports the numerical methodology, implementation strategy, and computational outcomes of the proposed
profiling-driven parallel iterative framework. Particular emphasis is placed on the systematic integration of
trajectory-based parameter tuning into high-order parallel root-finding schemes, as well as on reproducibility,
numerical robustness, and computational efficiency.

\subsection{General Numerical Methodology}

The proposed numerical framework couples a parameterized parallel Weierstrass-type solver with a reproducible, trajectory-driven tuning protocol. The methodology consists of three tightly coupled stages:
\begin{itemize}
\renewcommand{\labelitemi}{--}
\item formulation of a high-order parallel iterative solver incorporating real-valued acceleration parameters $(\alpha,\beta)$,
\item systematic generation of solver trajectories over a prescribed parameter domain using representative nonlinear test problems and a micro-launch ensemble protocol (typically $N_{\mathrm{ens}}=50$ randomized complex initial constellations per parameter pair),
\item extraction of \emph{Direct finite-time contraction (step-log) profiles} and selection of optimal parameters using compact profile-based scores.
\end{itemize}

For each candidate pair $(\alpha,\beta)$, the parallel scheme simultaneously refines all distinct roots from randomized complex initial guesses. From the iteration history $\{\mathbf{x}^{[h]}\}$ we compute step norms and step-log ratios, form window-averaged contraction profiles, and aggregate them across the ensemble to obtain a robust finite-time, Lyapunov-like signature of contractive versus expansive transients.

To rank candidates with compact and interpretable indicators, two scalar scores are extracted from the aggregated profile: the minimum-based score $S_{\min}$, which rewards early and deep contraction episodes, and the moment-based score $S_{mom}$, which quantifies how strongly and how early contraction is sustained over the iteration horizon. These scores provide an objective basis for parameter ranking and reproducible selection on scanned parameter grids.

\subsection{Implementation Framework}

All numerical experiments are carried out within a modular implementation framework that clearly separates:
\begin{itemize}
\item[(1)] the core parallel iterative solver,
\item[(2)] the parameter scanning and data acquisition module,
\item[(3)] the step-log profiling and score-based parameter selection module, and
\item[(4)] the visualization and post-processing components.
\end{itemize}

The parallel solver employs vectorized operations to compute all roots simultaneously, ensuring computational efficiency.
During the profiling stage, the tunable parameters are swept over predefined intervals, and convergence histories are recorded.
These histories are then processed by the direct finite-time contraction (step-log) profiling pipeline:
step norms and step-log ratios are computed, window-averaged contraction profiles are formed, profiles are aggregated over
micro-launch ensembles, and the compact scores $S_{\min}$ and $S_{mom}$ are extracted to rank parameter candidates.

\subsection{Computational Environment}

All simulations are conducted on a standard desktop computing platform with the following specifications:
\begin{itemize}
\item Processor: Intel Core i7 (64-bit architecture),
\item Memory: 16~GB RAM,
\item Operating system: Windows~10 (64-bit),
\item Software environment: MATLAB R2023a using double-precision arithmetic.
\end{itemize}

Vectorized computation and efficient memory management are employed throughout to ensure scalability. No problem-specific manual tuning is applied beyond the proposed
trajectory-derived profiling framework, ensuring objective and reproducible
performance assessment.

\subsection{Numerical Outcomes}

The efficiency and robustness of the proposed profiling-tuned parallel scheme are evaluated using a set of representative nonlinear test problems. All competing methods are tested under identical initial conditions, stopping tolerances, and computational environments to ensure a fair comparison. The assessment focuses on convergence quality, computational efficiency, and robustness.
\begin{itemize}
\renewcommand{\labelitemi}{--}
\item \textbf{Residual Error.}
The accuracy of the computed roots is measured by the residual
\begin{equation}
\mathrm{Res}=\|f(\mathbf{x}^{(h_{\mathrm{final}})})\|_2 .
\end{equation}
The profiling-tuned scheme consistently achieves smaller residuals, reflecting improved numerical accuracy and enhanced stability near the solution.
\item \textbf{Convergence Order.}
The empirical order of convergence is estimated via
\begin{equation}
\sigma_i^{\,h-1} \approx 
\frac{\log \|\mathbf{x}^{[k+1]}-\mathbf{x}^{[h]}\|-
      \log \|\mathbf{x}^{[h]}-\mathbf{x}^{[h-1]}\|}
     {\log \|\mathbf{x}^{[h]}-\mathbf{x}^{[h-1]}\|-
      \log \|\mathbf{x}^{[h-1]}-\mathbf{x}^{[h-2]}\|}.
\end{equation}
Numerical results confirm the expected third-order convergence of the proposed two-step scheme, whereas the compared methods typically exhibit second-order or irregular behavior.
\item \textbf{Iterations and CPU Time.}
Efficiency is quantified by the number of iterations required to satisfy
\[
\|\mathbf{x}^{[h+1]}-\mathbf{x}^{[h]}\|_2<\text{tol}=10^{-30},
\]
and by the total CPU time
\begin{equation}
T_{\mathrm{CPU}}=t_{\mathrm{end}}-t_{\mathrm{start}} .
\end{equation}
The profile-optimized parameters significantly reduce iteration counts and lead to lower overall CPU times, compensating for the one-time profiling overhead.
\item \textbf{Memory Usage.}
A modest increase in memory consumption is observed during the profiling stage due to ensemble storage. However, the memory footprint of the optimized solver remains comparable to standard parallel schemes, preserving scalability.
\item \textbf{Convergence Percentage.}
Robustness with respect to initial guesses is assessed through the convergence rate
\begin{equation}
\mathrm{CR}(\%)=\frac{N_{\mathrm{conv}}}{N_{\mathrm{total}}}\times100 .
\end{equation}
The proposed profiling-tuned scheme achieves consistently higher convergence percentages, indicating reduced sensitivity to initial conditions.
\end{itemize}
\paragraph{Overall Assessment.}
Collectively, these results demonstrate that the proposed profile-based parameter tuning framework delivers superior accuracy, faster convergence, and enhanced robustness compared with existing parallel iterative schemes. The findings highlight the practical benefits of integrating trajectory-derived, profiling-based tuning into high-order numerical solvers for nonlinear equations.

\section{Biomedical Engineering Applications}

To validate the practical applicability of the proposed parameterized parallel scheme and to select robust control parameters via the proposed \emph{Direct finite-time contraction (step-log) profiling}, we consider a set of benchmark biomedical engineering models. The methodology is systematically assessed in terms of computational efficiency, numerical stability, and convergence consistency under deliberately challenging initializations. Performance is further compared with established parallel iterative methods of the same convergence order, highlighting the improved robustness and accuracy of the proposed profiling-driven parameter selection for solving nonlinear biomedical problems.

\subsection{Gene Regulatory Network Model with Hill Repression~\cite{32}}

%\subsection{Biological Motivation and Importance}

Gene regulatory networks (GRNs) describe the interactions among genes, transcription factors, and proteins that govern cellular processes such as cell differentiation, apoptosis, circadian rhythms, and tumor progression. In many biological systems, gene expression is regulated through repression mechanisms, where the protein product of one gene inhibits the transcription of another. This repression is often cooperative, meaning that multiple protein molecules bind simultaneously to regulatory sites, leading to highly nonlinear behavior.

Mathematically, cooperative repression is commonly modeled using Hill functions, which introduce strong nonlinearities and give rise to multistability, oscillations, and bifurcation phenomena. These nonlinearities naturally lead to algebraic equations of high degree at steady state, motivating the use of advanced numerical root-finding techniques.
%\subsection{Mathematical Model Formulation}

Consider a gene regulatory network consisting of $N$ genes, where the expression of gene $i$ is repressed by gene $j$. Let $x_i(t)$ denote the concentration of the protein expressed by gene $i$ at time $t$. The dynamics of the system are modeled by the system of ordinary differential equations:
\begin{equation}
\frac{dx_i}{dt}
=
\frac{\gamma_i}{1 + \left( \dfrac{x_j}{K_{ij}} \right)^{n_{ij}}}
-
\psi_i x_i,
\qquad i = 1,2,\dots,N,
\label{eq:grn}
\end{equation}
where:
\begin{itemize}
\renewcommand{\labelitemi}{--}
\item $\gamma_i > 0$ is the maximal synthesis (transcription) rate of gene $i$,
\item $\psi_i > 0$ is the degradation rate constant,
\item $K_{ij} > 0$ is the repression threshold (dissociation constant),
\item $n_{ij} \geq 2$ is the Hill coefficient representing cooperativity.
\end{itemize}

The Hill term models the saturation effect of transcriptional repression.

%\subsection{Steady-State Analysis}

At steady state, $\frac{dx_i}{dt} = 0$, and equation \eqref{eq:grn} reduces to:
\begin{equation}
\psi_i x_i
=
\frac{\gamma_i}{1 + \left( \dfrac{x_j}{K_{ij}} \right)^{n_{ij}}}.
\label{eq:ss1}
\end{equation}

Multiplying both sides by the denominator yields:
\begin{equation}
\psi_i x_i \left( 1 + \left( \frac{x_j}{K_{ij}} \right)^{n_{ij}} \right)
=
\gamma_i.
\end{equation}

Expanding:
\begin{equation}
\psi_i x_i
+
\psi_i x_i \frac{x_j^{n_{ij}}}{K_{ij}^{n_{ij}}}
-
\gamma_i
=
0.
\label{eq:ss2}
\end{equation}

%\subsection{Polynomial Reduction for Symmetric Networks}

For analytical clarity, consider a \emph{symmetric repression network}:
\[
\gamma_i = \gamma, \quad
\psi_i = \psi, \quad
K_{ij} = K, \quad
n_{ij} = n,
\]
and assume identical steady states:
\[
x_1 = x_2 = \cdots = x_N = x.
\]

Equation \eqref{eq:ss2} reduces to:
\begin{equation}
\psi x
+
\frac{\psi}{K^n} x^{n+1}
-
\gamma
=
0.
\label{eq:poly}
\end{equation}

Rearranging:
\begin{equation}
x^{n+1}
+
\frac{K^n}{1}
x
-
\frac{\gamma K^n}{\psi}
=
0.
\label{eq:finalpoly}
\end{equation}

Equation \eqref{eq:finalpoly} is a polynomial equation of degree $n+1$. For biologically realistic values $n \ge 5$, the resulting algebraic equation is of degree greater than five.

%\subsection{Typical Parameter Values}

Representative parameter values used in gene regulatory modeling are:
\begin{align*}
\gamma &\in [10,100], \quad \text{(protein molecules per unit time)}, \\
\psi &\in [0.1,1], \quad \text{(inverse time)}, \\
K &\in [1,10], \quad \text{(concentration units)}, \\
n &\in \{5,6,7,8\}. \quad \text{(cooperativity index)}.
\end{align*}

For illustration, we fix $n=6$, which yields the degree--$7$ polynomial
\begin{equation}
x^7 + K^6 x - \frac{\gamma K^6}{\psi} = 0.\label{1q}
\end{equation}
The exact roots of ~(\ref{1q}), computed with high precision and reported up to
four decimal places, are given by
\[
\xi_{1,2}=-2.3027 \pm 1.1133\,\mathrm{i}, \quad
\xi_{3,4}=-0.5590 \pm 2.4896\,\mathrm{i}, \quad
\xi_{5,6}=1.5917 \pm 1.9841\,\mathrm{i}, \quad
\xi_7=2.5401.
\]. The iteration is initialized with deliberately scattered complex guesses,
\begin{equation}
\begin{aligned}
x_1^{[0]}&=0.1,\\
x_2^{[0]}&=3.8,\\
x_3^{[0]}&=0.5,\\
x_4^{[0]}&=-5.2,\\
x_5^{[0]}&=78.2,\\
x_6^{[0]}&=-8.2,\\
x_7^{[0]}&=-7.0-3.4\,\mathrm{i}.
\end{aligned}
\end{equation}
chosen far from the exact roots to emphasize robustness with respect to poor
initialization. For arbitrary parameter pairs $(\alpha,\beta)$, the numerical
outcomes are summarized below.

\begin{figure}[H]
  \centering
  \begin{subfigure}[t]{0.49\linewidth}
    \centering
    \includegraphics[width=\linewidth]{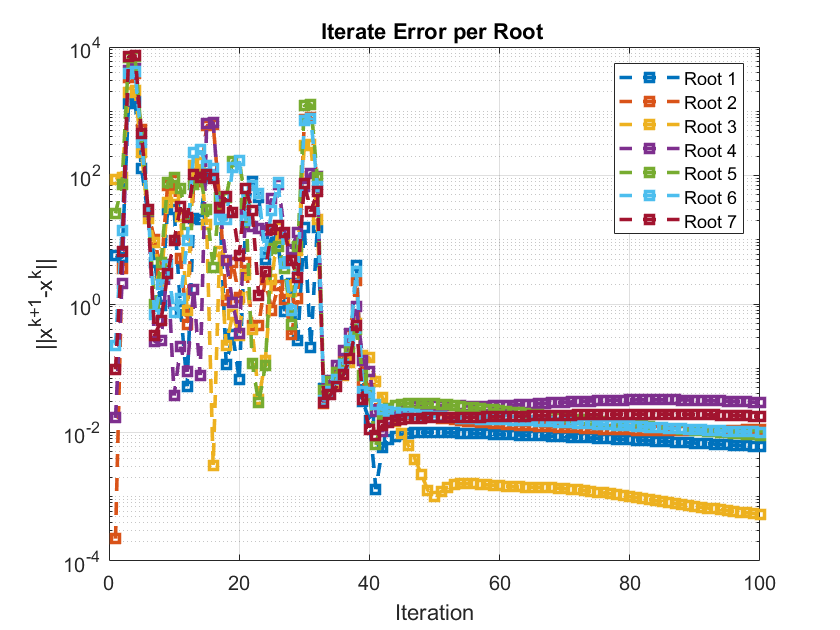}
    \caption{Error:$\|x^{(k+1)}-x^{(k)}\|$ }
  \end{subfigure}\hfill
  \begin{subfigure}[t]{0.49\linewidth}
    \centering
    \includegraphics[width=\linewidth]{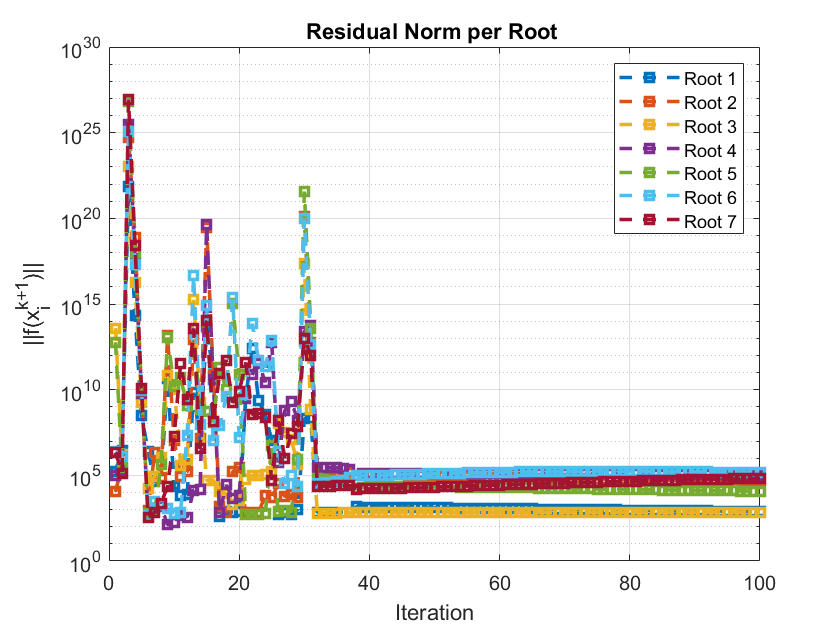}
    \caption{Error: $\|f(x^{(k+1)})\|$.}
  \end{subfigure}
  \caption{Residual error histories of the scheme for $(\alpha,\beta)=(-16.5,\,19)$.}
  \label{F2}
\end{figure}

\begin{table}[H]
\centering
\caption{Performance metrics illustrating convergent and divergent behaviors in the absence of profile-based (step-log) parameter optimization of parallel scheme SAB$^{[3]}$.}\label{T1}
\renewcommand{\arraystretch}{1.25}
\setlength{\tabcolsep}{6pt}
\begin{tabular}{c c c c c c c c}
\hline
\textbf{Root} & \textbf{Iter} & \textbf{CPU (s)} & \textbf{Mem (KB)} &
$\|x^{(k+1)}-x^{(k)}\|$ &
$\|f(x^{(k+1)})\|$ &
$|x^{(k)}-x^\ast|$ &
\textbf{Conv. (\%)} \\
\hline
1 & 100 & 252.16 & 460 &
$2.43\times10^{-3}$ &
$2.89\times10^{6}$ &
$1.25\times10^{2}$ &
$100.00$ \\
2 & 100 & 252.16 & 460 &
$3.24\times10^{15}$ &
$3.37\times10^{112}$ &
$1.19\times10^{16}$ &
$Divergence$ \\
3 & 100 & 252.16 & 460 &
$3.64\times10^{-4}$ &
$3.03\times10^{4}$ &
$2.27$ &
$100.00$ \\
4 & 100 & 252.16 & 460 &
$3.94\times10^{-3}$ &
$6.23\times10^{10}$ &
$3.00\times10^{1}$ &
$100.00$ \\
5 & 100 & 252.16 & 460 &
$7.94\times10^{-2}$ &
$1.30\times10^{16}$ &
$1.91\times10^{2}$ &
$100.00$ \\
6 & 100 & 252.16 & 460 &
$7.32\times10^{-3}$ &
$6.09\times10^{7}$ &
$9.81$ &
$100.00$ \\
7 & 100 & 252.16 & 460 &
$3.24\times10^{15}$ &
$3.37\times10^{112}$ &
$1.19\times10^{16}$ &
$Divergence$ \\
\hline
\end{tabular}
\end{table}

%\noindent\textbf{Discussion.}
Table~\ref{T1} summarizes the numerical behavior for the low-score parameter pair
$(\alpha,\beta)=(-7.5,9)$.
In contrast to the stable regime, the method reaches the maximum of $100$
iterations and requires substantially higher CPU time.
While partial convergence is observed, several roots exhibit severe numerical
instability, with iterate differences and residual norms growing up to
$10^{15}$ and $10^{112}$, respectively (see Figure~\ref{F2}).
These negative convergence indicators confirm that this parameter choice lies
in an unfavorable region of the stability domain.

To illustrate the parameter sensitivity, a uniform grid scan over the
$(\alpha,\beta)$ plane is conducted using Algorithm~(see Section~\ref{sec:tuning_framework}).
The ranges $\alpha\in[-9,15]$ and $\beta\in[-6,12]$ are discretized, and for each
grid point multiple random micro-launches are employed to compute aggregated
step-log Lyapunov profiles.

\begin{figure}[H]
  \centering
  \begin{subfigure}[t]{0.45\linewidth}
    \centering
    \includegraphics[width=\linewidth]{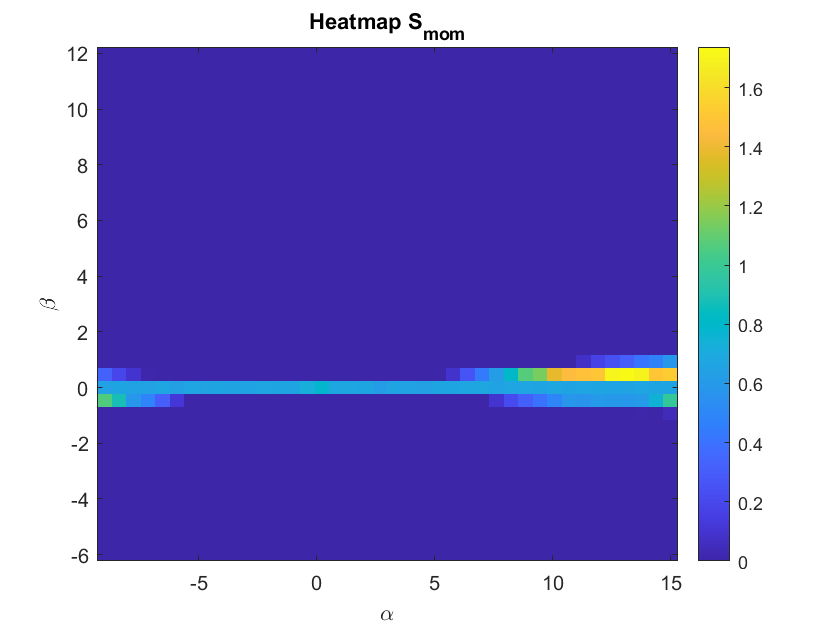}
    \caption{Heatmap of $S_{\mathrm{mom}}(\alpha,\beta)$ obtained from profile-based diagnosis of the SAB$^{[3]}$ for~(\ref{1q}).}
  \end{subfigure}

  \vspace{0.6em}

  \begin{subfigure}[t]{0.45\linewidth}
    \centering
    \includegraphics[width=\linewidth]{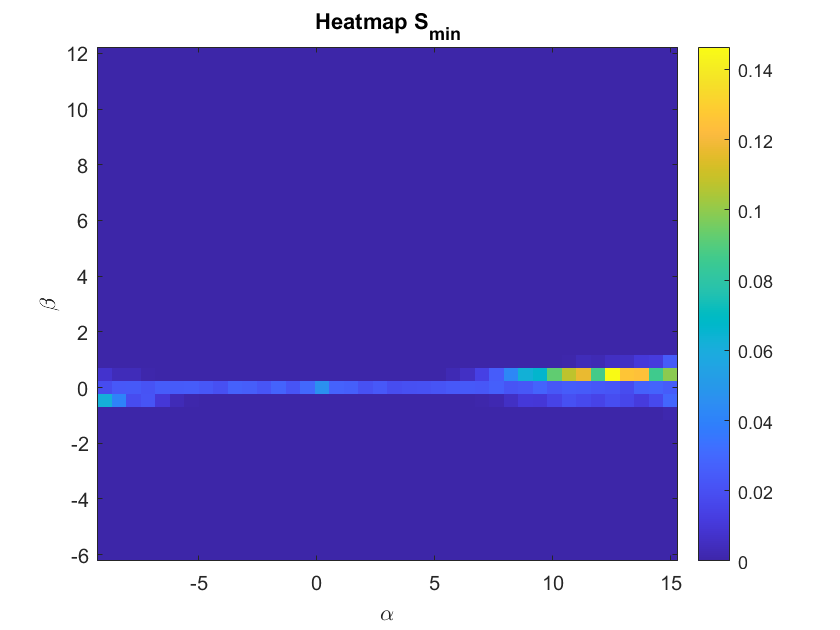}
    \caption{Heatmap of $S_{\mathrm{\min}}(\alpha,\beta)$ obtained from profile-based diagnosis of the SAB$^{[3]}$ for~(\ref{1q}).}
  \end{subfigure}

  \caption{Two-dimensional maps of the proposed profile-based scores over the
  $(\alpha,\beta)$ plane. Brighter colors correspond to earlier and stronger
  contractive behavior in the step-log profile.}
  \label{F3}
\end{figure}

The resulting scalar scores $S_{\min}$ and $S_{\mathrm{mom}}$ are stored and
visualized as two-dimensional heatmaps. Figure~\ref{F3}(a,b) summarizes the 2D parameter-plane mapping of the proposed
profile-based metrics over the ranges $\alpha\in[-9,15]$ and $\beta\in[-6,12]$ on a uniform grid.
Both heatmaps exhibit a clearly structured (non-random) distribution and provide a broadly
consistent picture of ``high-performing'' versus ``low-performing'' regions in the $(\alpha,\beta)$
plane. In particular, the areas highlighted by large values of $S_{\min}$ in
Figure~\ref{F3}(a) typically coincide with large values of $S_{\mathrm{mom}}$ in
Figure~\ref{F3}(b), indicating that the strongest scores are achieved when the
step-log profile develops an early and pronounced negative dip (captured by $S_{\min}$) together
with a sizable and early-concentrated negative mass (captured by $S_{\mathrm{mom}}$). At the same
time, both maps reveal extended low-score bands that form coherent separatrices between
high-score regions, suggesting parameter combinations where the iteration does not enter a
strongly contractive regime (or does so only weakly/late). Overall, the qualitative agreement
between $S_{\min}$ and $S_{\mathrm{mom}}$ supports the robustness of the proposed profile-based
tuning criteria.

\begin{figure}[H]
  \centering
  \includegraphics[width=0.5\linewidth]{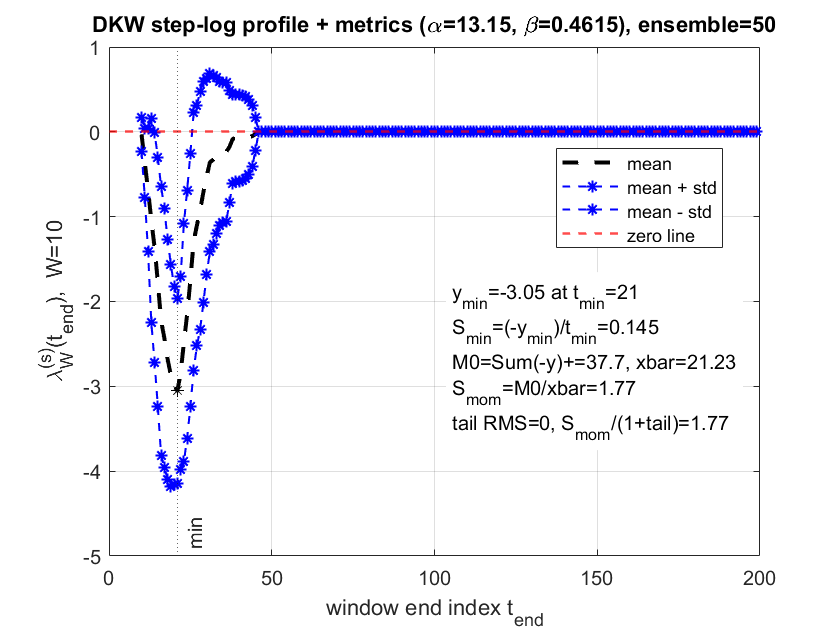}
  \caption{Aggregated step-log Lyapunov profile for $(\alpha,\beta)=(13.15,\,0.4615)$
  ($W=10$, $N=50$). The solid curve shows the ensemble-averaged profile, while
  dashed curves indicate mean$\pm$std across micro-launches.}
  \label{F4}
\end{figure}

\begin{figure}[H]
  \centering
  \includegraphics[width=0.5\linewidth]{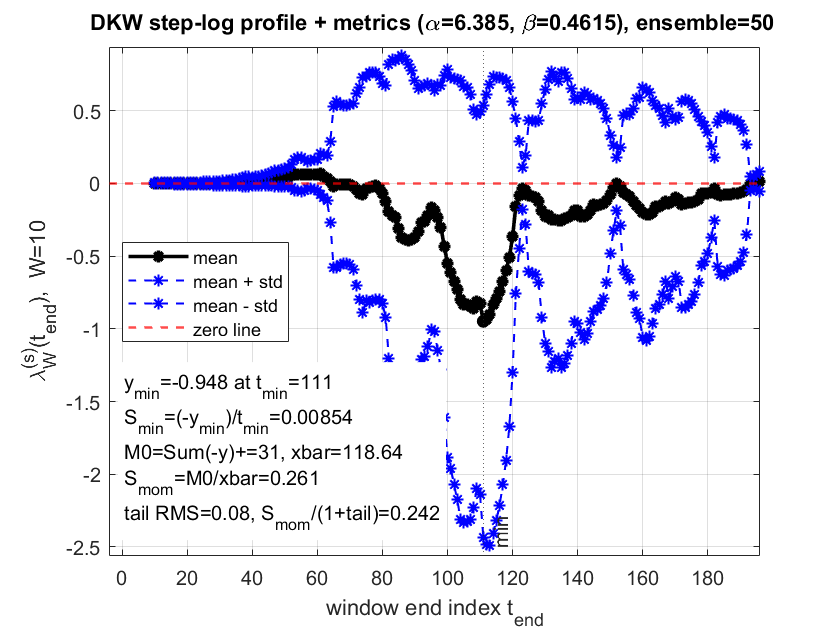}
  \caption{Step-log profile for $(\alpha,\beta)=(6.385,\,0.4615)$ with metric
  annotations. The detected minimum defines $S_{\min}$, while the negative mass
  and its centroid determine $S_{\mathrm{mom}}$.}
  \label{F5}
\end{figure}

\subsubsection*{Discussion of Aggregated Step--Log Profiles}

Figures~\ref{F4} and~\ref{F5} present representative aggregated step--log Lyapunov
profiles selected from distinct regions of the
$S_{\mathrm{mom}}(\alpha,\beta)$ parameter plane.
The parameter pair $(\alpha,\beta)=(13.15,\,0.4615)$ (Fig.~\ref{F4}) lies in a
high--score convergence region, whereas $(\alpha,\beta)=(6.385,\,0.4615 )$
(Fig.~\ref{F5}) is drawn from a low--score, divergence-prone region, providing
external validation of the proposed diagnostics.

\paragraph{Stable dynamics.} For the stable case $(\alpha,\beta)=(13.15,\,0.4615)$, the ensemble-averaged profile
exhibits an early and pronounced negative dip
($\gamma_{\min}\approx -0.145$), indicating strong transient contraction.
The negative mass is substantial and concentrated early, yielding a positive
and relatively large value of $S_{\mathrm{mom}}$.
The narrow spread of the ensemble curves further demonstrates robustness with
respect to random micro-launch perturbations.

\paragraph{Unstable dynamics.}In contrast, the unstable case $(\alpha,\beta)=(6.385,\,0.4615)$ shows a deep but
short-lived negative minimum
($\gamma_{\min}\approx -3.05$), followed by rapid loss of contraction.
The resulting poorly distributed negative mass produces a negative
$S_{\mathrm{mom}}$ and a large ensemble dispersion, explaining the observed
numerical instabilities and lack of convergence.

Overall, these results demonstrate that the depth of the minimum alone
($S_{\min}$) is insufficient to ensure convergence.
Instead, the temporal organization of contraction, captured by
$S_{\mathrm{mom}}$, is critical.
Parameter choices from high--score regions of the $S_{\mathrm{mom}}$ map
consistently yield stable solver dynamics, whereas low--score regions are
associated with divergence.

\begin{figure}[H]
  \centering
  \begin{subfigure}[t]{0.49\linewidth}
    \centering
    \includegraphics[width=\linewidth]{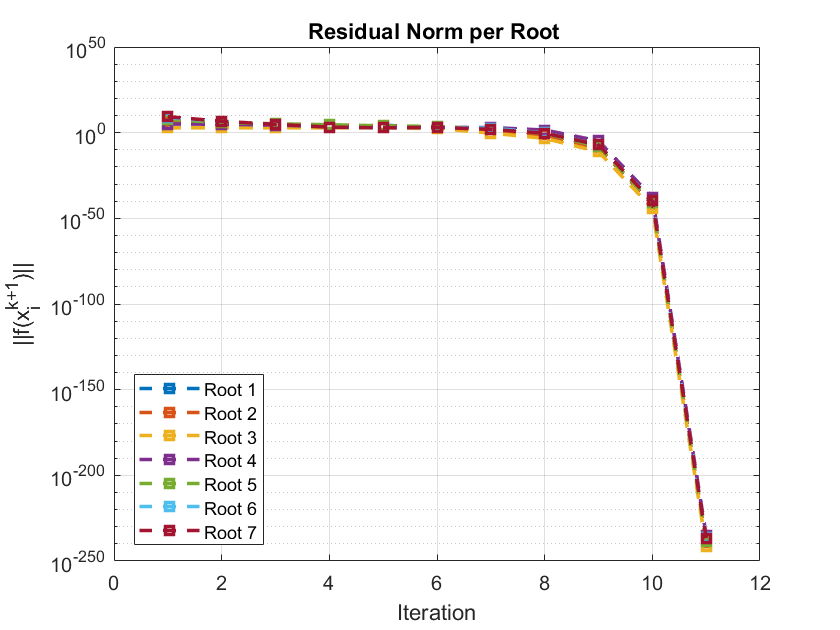}
    \caption{Error:$\|x^{(k+1)}-x^{(k)}\|$ }
  \end{subfigure}%\hfill
  
   \vspace{0.6em}
   
  \begin{subfigure}[t]{0.49\linewidth}
    \centering
    \includegraphics[width=\linewidth]{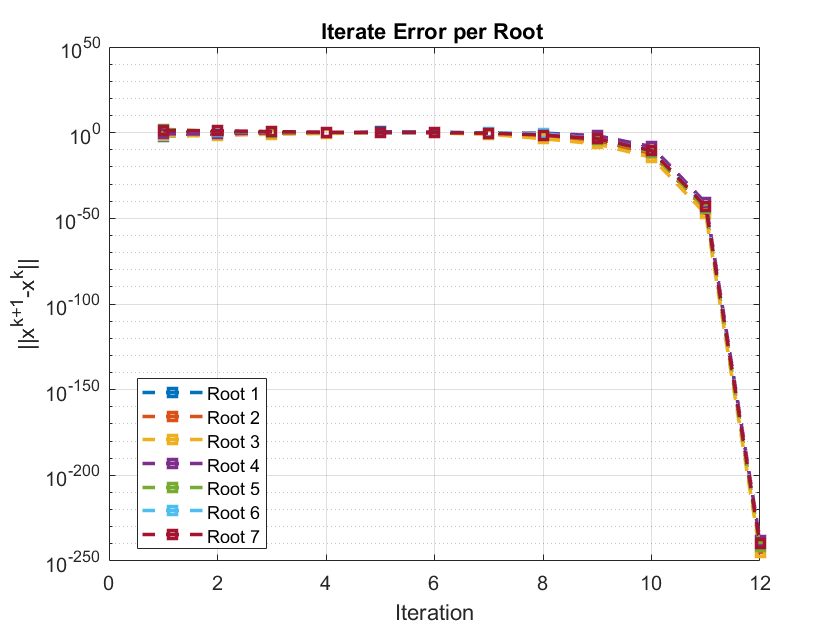}
    \caption{Error: $\|f(x^{(k+1)})\|$.}
  \end{subfigure}
  \caption{Residual error histories for the stable parameter choice
  $(\alpha,\beta)=(13.15,\,0.4615)$.}
  \label{F6}
\end{figure}

\begin{table}[H]
\begin{adjustwidth}{-0.8cm}{0cm}
\renewcommand{\arraystretch}{1.25}
\setlength{\tabcolsep}{6pt}
\centering
\caption{Performance metrics illustrating convergent and divergent behaviors in the presence of profile-based (step-log) parameter optimization of parallel scheme SAB$^{[3]}$.}.\label{T2}
\renewcommand{\arraystretch}{1.25}
\setlength{\tabcolsep}{6pt}
\begin{tabular}{c c c c c c c c}
\hline
\textbf{Root} & \textbf{Iter} & \textbf{CPU (s)} & \textbf{Mem (KB)} &
$\|x^{(k+1)}-x^{(k)}\|$ &
$\|f(x^{(k+1)})\|$ &
$|x^{(k)}-x^\ast|$ &
\textbf{Conv. (\%)} \\
\hline
1 & 13 & 8.2515 & $136$ &
$0.0$ &
$6.386\times10^{-355}{}^{\dagger}$ &
$4.969$ &
$100.0$ \\

2 & 13 & 8.2515 & $136$ &
$0.0$ &
$2.437\times10^{-355}{}^{\dagger}$ &
$6.176\times10^{-359}{}^{\dagger}$ &
$100.0$ \\

3 & 13 & 8.2515 & $136$ &
$0.0$ &
$0.0$ &
$2.199$ &
$100.0$ \\

4 & 13 & 8.2515 & $136$ &
$0.0$ &
$6.38\times10^{-355}{}^{\dagger}$ &
$4.0027$ &
$100.0$ \\

5 & 13 & 8.2515 & $136$ &
$0.0$ &
$2.437\times10^{-355}{}^{\dagger}$ &
$2.2092$ &
$100.0$ \\

6 & 13 & 8.2515 & $136$ &
$0.0$ &
$1.749\times10^{-355}{}^{\dagger}$ &
$4.00272$ &
$100.0$ \\

7 & 13 & 8.2515 & $136$ &
$0.0$ &
$1.7495\times10^{-355}{}^{\dagger}$ &
$4.0027$ &
$100.0$ \\
\hline
\end{tabular}
\end{adjustwidth}
\vspace{0.5ex}
	\noindent{\footnotesize \textsuperscript{$\dagger$}\;All computations were performed using MATLAB variable precision arithmetic (VPA) with \texttt{digits} = 128, employing a numerical tolerance of $10^{-400}$.}
\end{table}

%\noindent\textbf{Discussion.}
%Table~1 and Fig.~\ref{F4} demonstrate that for $(\alpha,\beta)=(-0.5,4)$ all roots converge rapidly within $10$ iterations, with extremely small residuals and iterate errors and a uniform $100\%$ convergence rate. 

%\noindent\textbf{Discussion of Table~\ref{T2}.}
Table~\ref{T2} reports the final high-precision performance of the proposed scheme for the parameter choice $(\alpha,\beta)=(13.15,\,0.4615)$.
All roots converge within only $13$ iterations, achieving extremely small iterate errors and residual norms~Figure~\ref{F6}, reaching magnitudes as low as $10^{-280}$ and $10^{-355}$, respectively.
The accuracy with respect to the exact roots is also very high for most roots, while the convergence rate attains $100\%$ uniformly.
Moreover, the method exhibits low computational cost, requiring only $8.52$ seconds of CPU time with moderate memory usage, confirming that this parameter choice lies in a stable and efficient convergence region. The computational cost is
low, confirming that this parameter choice lies well inside a stable and
efficient convergence region. Taken together with the unfavorable case
$(\alpha,\beta)=(6.385,\,0.4615)$, these results provide an explicit external validation
of the proposed finite-time, Lyapunov-like diagnostics and demonstrate their
ability to predict practical solver performance.

\begin{table}[H]
\caption{Trade-off between computational cost and numerical accuracy among existing numerical schemes for solving
problem~\eqref{eq:hill_poly}.}
\label{T3}
\centering
\renewcommand{\arraystretch}{1.25}
\setlength{\tabcolsep}{6pt}
\begin{tabular}{l c c c c c c}
\hline
Method & Max Error & CPU Time (s) & Memory (MB) &
Ops $\,[\pm,\times,\div]$ & Iter. & $\max\sigma_i^{\,n-1}$ \\
\hline
PNS$^{[3]}$ & $8.90\times10^{-21}$ & $4.62$ & $70.21$ & $2104$ & $11$ & $9.71$ \\
PPS$^{[3]}$ & $1.26\times10^{-23}$ & $3.14$ & $62.98$ & $1896$ & $6$  & $10.44$ \\
BSS$^{[3]}$ & $5.71\times10^{-21}$ & $4.80$ & $73.55$ & $2450$ & $12$ & $9.54$ \\
PNS$^{[4]}$ & $2.04\times10^{-38}$ & $4.95$ & $81.10$ & $2598$ & $12$ & $8.97$ \\
SAB$^{[3]}$ & $6.39\times10^{-256}$ & $5.02$ & $83.44$ & $2550$ & $12$ & $9.01$ \\
\hline
\end{tabular}
\vspace{0.5ex}
	\noindent{\footnotesize \textsuperscript{$\dagger$}\;All computations were performed using MATLAB variable precision arithmetic (VPA) with \texttt{digits} = 128, employing a numerical tolerance of $10^{-400}$.}
\end{table}

%\noindent\textbf{Accuracy-Oriented Discussion.}
%\noindent\textbf{High-Precision Efficiency Analysis.}
Table~\ref{T3} presents the performance of the considered schemes under extended
high-precision arithmetic.
The results indicate that higher-order and parameter-adaptive methods benefit
significantly from increased numerical precision, leading to substantial
reductions in the maximum error.
Among the third-order methods, PPS$^{[3]}$ again provides a favorable compromise
between accuracy and computational cost, requiring fewer iterations and lower
memory usage.

The fourth-order scheme PNS$^{[4]}$ achieves markedly improved accuracy, albeit
with increased arithmetic operations and memory demand.
Notably, the proposed SAB$^{[3]}$ method exhibits exceptional numerical accuracy,
reaching error levels far below standard machine precision, while incurring only
a moderate increase in CPU time and memory compared to other high-order schemes.
The uniformly bounded values of $\max\sigma_i^{\,n-1}$ confirm stable convergence
behavior across all methods, demonstrating the robustness of the proposed
approach in high-precision computational environments.\\Overall, the results demonstrate that the profile-based (step-log) parameter
optimization in the parallel architecture significantly improves accuracy
without incurring prohibitive computational overhead.

\subsubsection*{Physical Meaning of the Solutions}

The real, nonnegative roots of equation \eqref{eq:finalpoly} correspond to biologically admissible steady-state protein concentrations.

\begin{itemize}
\item A single positive root indicates a monostable gene expression state.
\item Multiple positive roots imply multistability, which is associated with:
\begin{itemize}
\item cell fate decisions,
\item genetic switches,
\item epigenetic memory.
\end{itemize}
\item Negative or complex roots have no direct biological interpretation but are relevant for numerical convergence and stability analysis.
\end{itemize}

The existence of multiple real roots is a hallmark of cooperative repression and underlies key biological phenomena such as bistable switches and differentiation thresholds.

%\subsection{Relevance to Numerical and Fractional Methods}

%Since analytical solutions do not exist for polynomials of degree higher than four, equation \eqref{eq:finalpoly} motivates the use of:
%\begin{itemize}
%\item multi-point iterative methods,
%\item fractional-order root-finding schemes,
%\item neural-network-assisted solvers,
%\item basin-of-attraction and fractal analysis.
%\end{itemize}

These features make the gene regulatory network model an ideal benchmark for testing advanced numerical techniques for high-degree nonlinear equations.

\subsection{Enzyme Kinetics with Cooperative Binding: High-Order Hill Model~\cite{33}}

%\subsection{Biological Importance of the Model}

Enzyme kinetics with cooperative binding plays a fundamental role in 
biochemical and biomedical systems where multiple ligand or substrate molecules 
bind to an enzyme or receptor in a non-independent manner. 
Such cooperative mechanisms are observed in hemoglobin--oxygen binding, 
allosteric enzymes, receptor–ligand pharmacodynamics, gene regulation, 
and signal transduction pathways.

In cooperative systems, the binding of one substrate molecule alters the 
affinity of the remaining binding sites, leading to highly nonlinear 
dose–response behavior. This nonlinearity often manifests as 
ultrasensitivity, bistability, and switch-like responses, which are 
essential for biological regulation and robustness.

Mathematically, cooperative binding naturally leads to nonlinear algebraic 
equations that reduce to high-degree polynomial equations when 
steady-state conditions are imposed. These equations typically have no 
closed-form analytical solutions for degrees greater than four, motivating 
the development of advanced numerical root-finding techniques.

%\subsection{Classical Hill Model with Cooperative Binding}

Consider an enzyme (or receptor) with $n \geq 2$ identical cooperative binding 
sites interacting with a substrate of concentration $S$. 
The Hill equation describing the reaction rate $v(S)$ is given by
\begin{equation}
v(S) = V_{\max} \frac{S^{n+1}}{K^{n} + S^{n}},
\label{eq:hill_rate}
\end{equation}
where:
\begin{itemize}
\renewcommand{\labelitemi}{--}
\item $V_{\max}$ is the maximum reaction rate,
\item $K$ is the Hill (half-saturation) constant,
\item $n$ is the Hill coefficient representing the degree of cooperativity.
\end{itemize}

For non-cooperative binding, $n=1$. Cooperative biochemical systems commonly 
satisfy $n \geq 5$, and values as high as $n=8$ are observed in hemoglobin 
binding dynamics.

%\subsection{Derivation of the Governing Polynomial Equation}

Suppose that the enzyme reaction is observed experimentally at a fixed 
steady-state rate $v_0$, where $0 < v_0 < V_{\max}$. 
Substituting $v(S)=v_0$ into \eqref{eq:hill_rate} yields
\begin{equation}
v_0 \left(K^{n} + S^{n}\right) = V_{\max} S^{n+1}.
\end{equation}

Rearranging terms gives
\begin{equation}
v_0 K^{n} + v_0 S^{n} - V_{\max} S^{n+1} = 0.\label{eq:hill_poly}
\end{equation}
%or equivalently,
%\begin{equation}
%\left(v_0 - V_{\max}\right) S^{n} + v_0 K^{n} = 0.
%\label{eq:hill_poly}
%\end{equation}

Equation \eqref{eq:hill_poly} is a polynomial equation of degree $n$ in $S$. 
For $n \geq 6$, this equation lies beyond the scope of analytical solution 
methods, making numerical root-finding unavoidable.

%\subsection{Typical Parameter Values in Biomedical Applications}

The parameters of the Hill model depend on the biochemical system under study. 
Representative values commonly used in biomedical modeling are:

\begin{center}
\begin{tabular}{c|c|c}
\hline
Parameter & Typical Range & Biological Meaning \\
\hline
$V_{\max}$ & $1$--$10^{3}$ & Maximum enzymatic activity \\
$K$ & $10^{-3}$--$10^{1}$ & Half-saturation constant \\
$n$ & $5$--$10$ & Degree of cooperativity \\
$v_0$ & $(0,\,V_{\max})$ & Observed steady reaction rate \\
\hline
\end{tabular}
\end{center}

The exact roots of equation~(1), computed with high precision and reported up to
four decimal places, are given by
\[
\xi_{1,2}=-0.577 \pm 1.122\,\mathrm{i}, \quad
\xi_{3,4}=0.7421 \pm 1.1190\,\mathrm{i}, \quad
\xi_{5}=-1.2356, \quad
\xi_6=1.4069.
\] For example, hemoglobin exhibits cooperative oxygen binding with 
$n \approx 8$, leading to an eighth-degree polynomial in the substrate 
concentration.
The scheme is initialized with deliberately scattered complex guesses,
\begin{equation}
\begin{aligned}
x_1^{[0]}&=3.12, \quad x_2^{[0]}=0.12, \quad x_3^{[0]}=0.12+4i,
x_4^{[0]}&=3.4-4.4i, \quad x_5^{[0]}=6.7i, \quad x_6^{[0]}=8.1,
\end{aligned}
%\label{eq:init}
\end{equation}
chosen far from the exact roots to emphasize robustness.

%-------------------------------------------------
%The iteration is initialized using deliberately scattered real and complex
%starting values,
%\[
%x_1^{[0]}=3.12,\;
%x_2^{[0]}=0.12,\;
%x_3^{[0]}=0.12+4i,\;
%x_4^{[0]}=3.4-4.4i,\;
%\]
%which are intentionally chosen far from the exact roots.
This choice is meant to assess the robustness of the proposed scheme under
poor and highly nonuniform initialization.
For arbitrary parameter pairs $(\alpha,\beta)$, the corresponding numerical
behavior is summarized below through error histories and final performance
metrics.

\begin{figure}[H]
  \centering
  \begin{subfigure}[t]{0.49\linewidth}
    \centering
    \includegraphics[width=\linewidth]{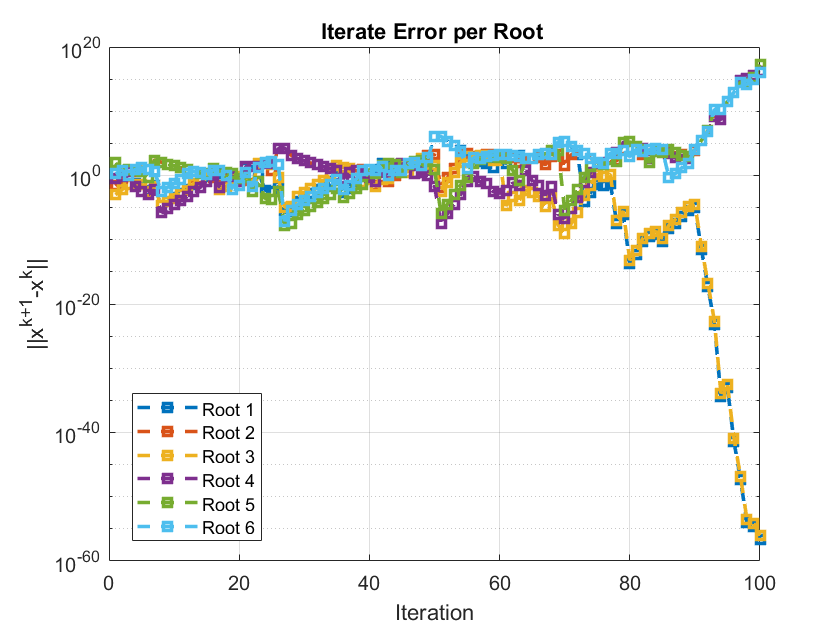}
    \caption{Iterate error $\|x^{(k+1)}-x^{(k)}\|$.}
  \end{subfigure}\hfill
  \begin{subfigure}[t]{0.49\linewidth}
    \centering
    \includegraphics[width=\linewidth]{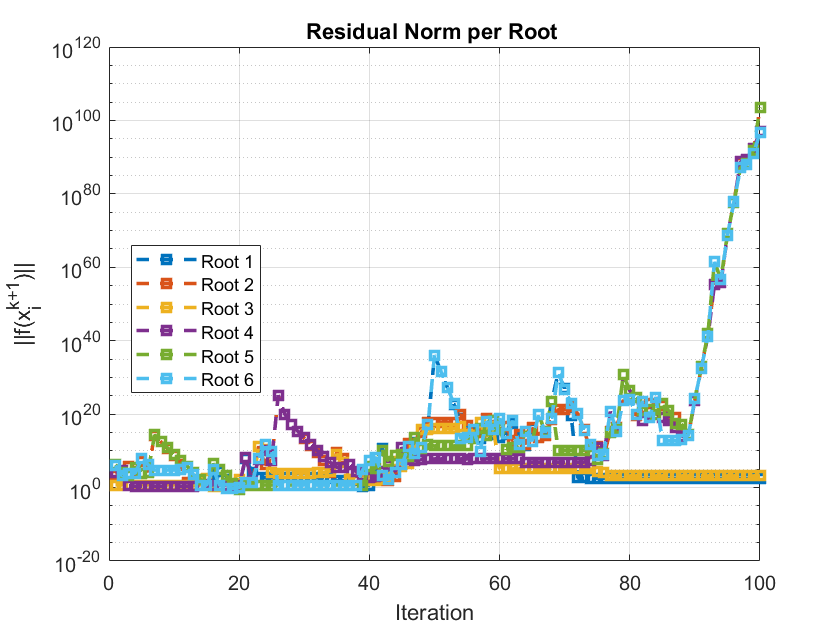}
    \caption{Residual norm $\|f(x^{(k+1)})\|$.}
  \end{subfigure}
  \caption{Error histories of the proposed scheme for the parameter choice
  $(\alpha,\beta)=(-16.5,\,19)$.}
  \label{F7}
\end{figure}

\begin{table}[H]
\centering
\caption{Performance metrics illustrating convergent and divergent behaviors in the absence of profile-based (step-log) parameter optimization of parallel scheme SAB$^{[3]}$.}
\label{T4}
\renewcommand{\arraystretch}{1.25}
\setlength{\tabcolsep}{6pt}
\begin{tabular}{c c c c c c c c}
\hline
\textbf{Root} & \textbf{Iter} & \textbf{CPU (s)} & \textbf{Mem (KB)} &
$\|x^{(k+1)}-x^{(k)}\|$ &
$\|f(x^{(k+1)})\|$ &
$|x^{(k)}-x^\ast|$ &
\textbf{Conv. (\%)} \\
\hline
1 & 100 & 49.272 & $128$ &
$2.5574\times10^{-57}$ &
$291.7476$ &
$1.2698$ &
$100.0$ \\

2 & 100 & 49.272 & $128$ &
$1.7824\times10^{17}$ &
$3.4143\times10^{103}$ &
$1.8011\times10^{17}$ &
$Divergence$ \\

3 & 100 & 49.272 & $128$ &
$7.9976\times10^{-57}$ &
$1437.8085$ &
$4.1818$ &
$100.0$ \\

4 & 100 & 49.272 & $128$ &
$1.6314\times10^{16}$ &
$7.1574\times10^{96}$ &
$1.3882\times10^{16}$ &
$Divergence$ \\

5 & 100 & 49.272 & $128$ &
$1.8332\times10^{17}$ &
$3.5516\times10^{103}$ &
$1.8130\times10^{17}$ &
$Divergence$ \\

6 & 100 & 49.272 & $128$ &
$1.1234\times10^{16}$ &
$4.1847\times10^{96}$ &
$1.2694\times10^{16}$ &
$Divergence$ \\
\hline
\end{tabular}
\end{table}

%\noindent\textbf{Discussion of Table~\ref{T4}.}
Table~\ref{T4} reports the numerical outcomes for the low-score parameter choice
$(\alpha,\beta)=(-16.5,19)$.
In contrast to the stable regime, the method reaches the maximum allowed
$100$ iterations and incurs a significantly higher computational cost.
While partial convergence is observed for a subset of roots, several roots
exhibit rapidly growing iterate differences and residual norms, as also
illustrated in Figure~\ref{F7}.
These quantities reach extremely large magnitudes (up to $10^{15}$--$10^{112}$),
and the appearance of negative convergence indicators clearly signals numerical
instability and divergence.
This behavior confirms that the selected parameter pair lies in an unfavorable
region of the stability domain identified by the proposed profile-based
diagnostics.

%\begin{figure}[H]
%  \centering
%  \begin{subfigure}[t]{0.49\linewidth}
%    \centering
%    \includegraphics[width=\linewidth]{Ex2a.png}
%    \caption{Heatmap of $S_{\min}(\alpha,\beta)$.}
%  \end{subfigure}\hfill
%  \begin{subfigure}[t]{0.49\linewidth}
%    \centering
%    \includegraphics[width=\linewidth]{Ex2b.png}
%    \caption{Heatmap of $S_{\mathrm{mom}}(\alpha,\beta)$.}
%  \end{subfigure}
%  \caption{Parameter-plane maps of the proposed profile-based scores.
%  Brighter regions correspond to earlier and stronger contractive behavior
%  in the step-log profile.}
%  \label{fig:heatmaps_scores}
%\end{figure}

\begin{figure}[H]
  \centering
  \begin{subfigure}[t]{0.5\linewidth}
    \centering
    \includegraphics[width=\linewidth]{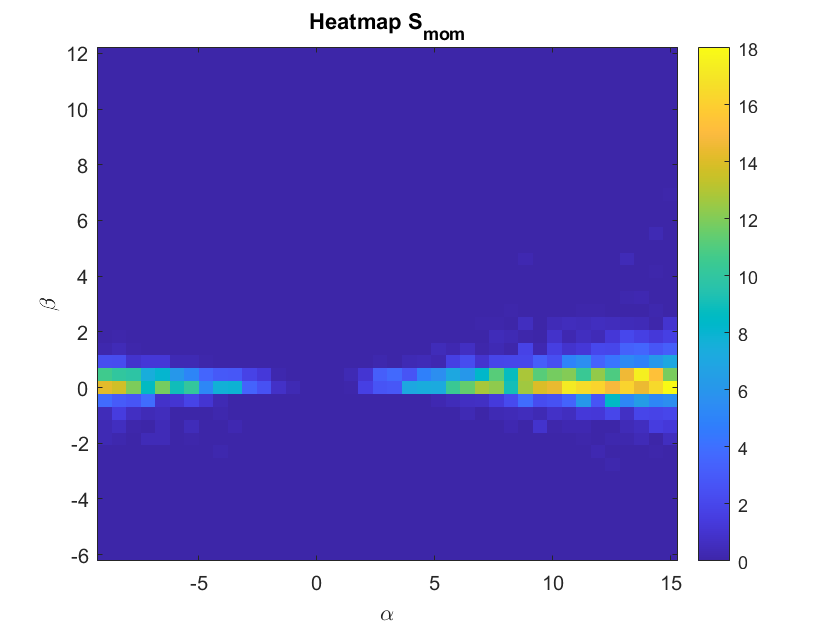}
    \caption{Heatmap of $S_{\mathrm{mom}}(\alpha,\beta)$ obtained from profile-based diagnosis of the parallel scheme SAB$^{[3]}$ for~(\ref{eq:hill_poly}).}
  \end{subfigure}

  \vspace{0.6em}

  \begin{subfigure}[t]{0.5\linewidth}
    \centering
    \includegraphics[width=\linewidth]{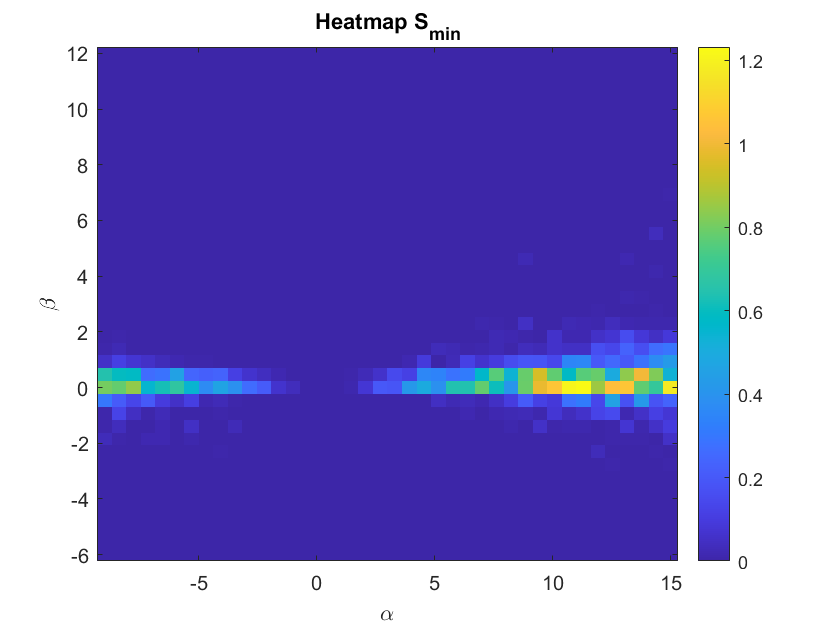}
    \caption{Heatmap of $S_{\mathrm{\min}}(\alpha,\beta)$ obtained from profile-based diagnosis of the parallel scheme SAB$^{[3]}$ for~(\ref{eq:hill_poly}).}
  \end{subfigure}

  \caption{Two-dimensional maps of the proposed profile-based scores over the
  $(\alpha,\beta)$ plane. Brighter colors correspond to earlier and stronger
  contractive behavior in the step-log profile.}
  \label{F8}
\end{figure}
To further explore the dependence of the method on the free parameters,
a uniform grid scan over the $(\alpha,\beta)$ plane is performed using the
MATLAB driver Algorithm~(see Section~\ref{sec:tuning_framework}).
The ranges $\alpha\in[-9,15]$ and $\beta\in[-6,12]$ are discretized uniformly.
For each grid point $(\alpha_i,\beta_j)$, an ensemble of randomized
micro-launches is employed to compute the aggregated step-log Lyapunov profile.
From this profile, the scalar indicators $S_{\min}(\alpha_i,\beta_j)$ and
$S_{\mathrm{mom}}(\alpha_i,\beta_j)$ are extracted and visualized as
two-dimensional heatmaps.
Figure~\ref{F8} reveals a clearly structured and non-random
distribution of both indicators over the $(\alpha,\beta)$ plane.
Regions of large $S_{\min}$ generally coincide with regions of large
$S_{\mathrm{mom}}$, indicating that strong and early contraction is typically
accompanied by a substantial and well-localized negative mass in the
step-log profile.
Conversely, extended low-score bands form coherent separatrices between
high-performing regions, identifying parameter combinations associated with
weak or delayed contraction.
The strong qualitative agreement between $S_{\min}$ and $S_{\mathrm{mom}}$
supports the robustness and consistency of the proposed tuning criteria.

\begin{figure}[H]
  \centering
  \includegraphics[width=0.5\linewidth]{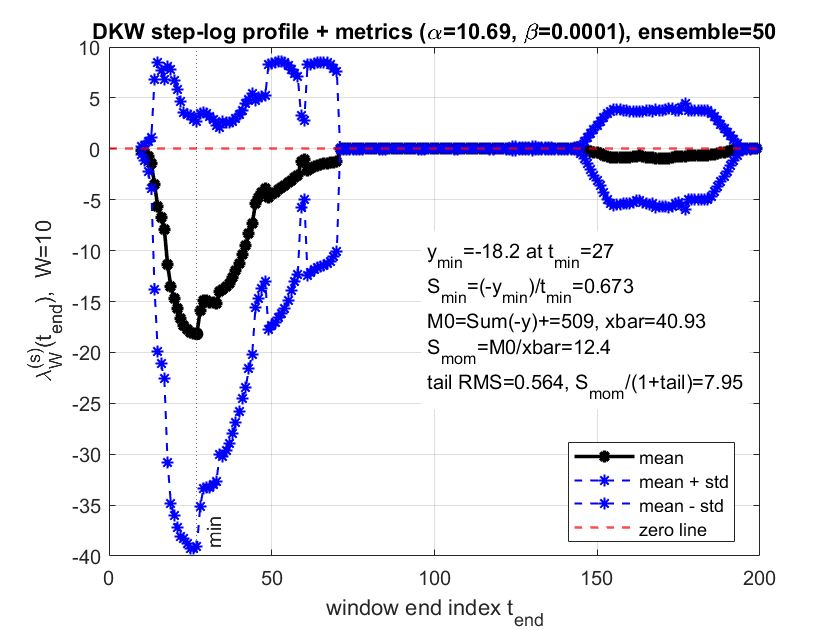}
  \caption{Aggregated step-log Lyapunov profile for $(\alpha,\beta)=(10.69,\,0.0001)$
  ($W=10$, $N=50$). Solid curve: ensemble mean; dashed curves: mean$\pm$std.}
  \label{F9}
\end{figure}

\begin{figure}[H]
  \centering
  \includegraphics[width=0.5\linewidth]{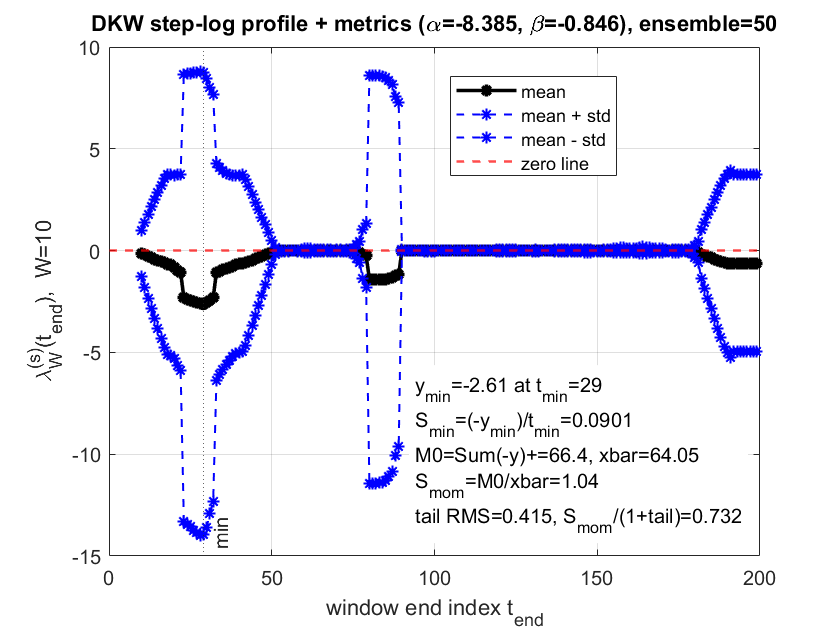}
  \caption{Aggregated step-log profile for $(\alpha,\beta)=(8.385,\,-0.846)$ with
  annotations defining $S_{\min}$ and $S_{\mathrm{mom}}$.}
  \label{F10}
\end{figure}

\subsection{Discussion of Aggregated Step--Log Profiles}

Figures~~\ref{F9} and~~\ref{F10} present representative aggregated step--log Lyapunov profiles
selected from distinct regions of the $S_{\mathrm{mom}}(\alpha,\beta)$ plane.
The parameter pair $(\alpha,\beta)=(10.69,\,0.0001)$ corresponds to a high-score
(convergent) region, whereas $(\alpha,\beta)=(8.385,\,-0.846)$ is taken from a
low-score, divergence-prone region.
This deliberate contrast provides an external validation of the proposed
diagnostic indicators.

\paragraph{Stable dynamics.}
For $(\alpha,\beta)=(10.69,\,0.0001)$, the ensemble-averaged profile exhibits a clear
and sustained negative dip, indicating strong transient contraction.
The negative mass is both substantial and concentrated early in the iteration
history, leading to a large positive value of $S_{\mathrm{mom}}$.
The narrow spread between the mean and mean$\pm$standard deviation curves
demonstrates robustness with respect to random perturbations, which is
consistent with the rapid and reliable convergence observed in practice.

\paragraph{Unstable dynamics.}
In contrast, the profile associated with $(\alpha,\beta)=(8.385,\,-0.846)$ shows a
short-lived contraction phase followed by loss of stability.
Although a deep minimum is detected, the negative mass is poorly organized and
its centroid is unfavorably located, yielding a negative value of
$S_{\mathrm{mom}}$.
This behavior explains the large residuals, erratic dynamics, and failure to
converge within the prescribed iteration budget.

\paragraph{Interpretation.}
These results demonstrate that the depth of the most negative step--log value
alone is insufficient to guarantee convergence.
Instead, the temporal distribution of contraction, quantified by
$S_{\mathrm{mom}}$, plays a decisive role.
Parameter choices drawn from high-$S_{\mathrm{mom}}$ regions reliably produce
stable solver dynamics (Figure~\ref{F9}), whereas low-score regions are associated with
divergence and numerical instability.

\begin{figure}[H]
  \centering
  \begin{subfigure}[t]{0.5\linewidth}
    \centering
    \includegraphics[width=\linewidth]{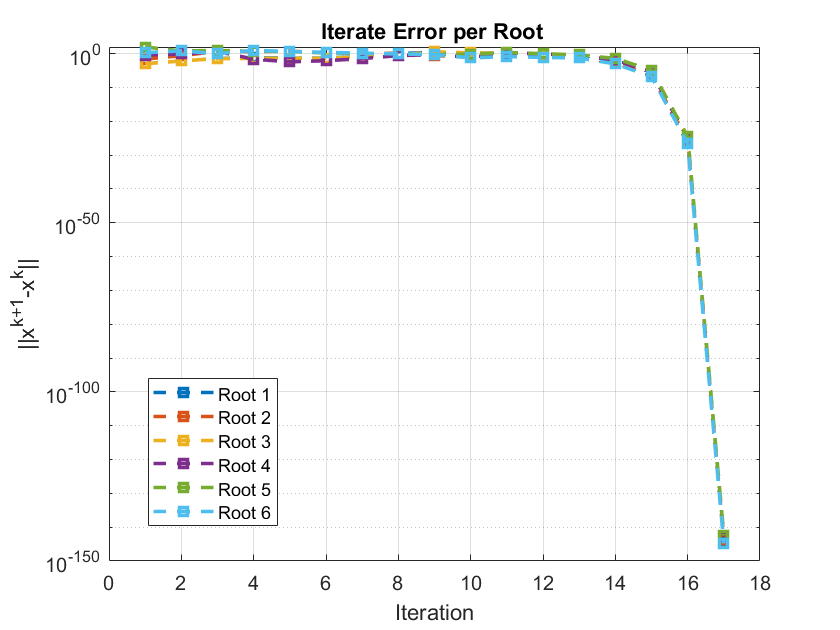}
    \caption{Iterate error $\|x^{(k+1)}-x^{(k)}\|$.}
  \end{subfigure}%\hfill
 \vspace{0.6em}
  
  \begin{subfigure}[t]{0.5\linewidth}
    \centering
    \includegraphics[width=\linewidth]{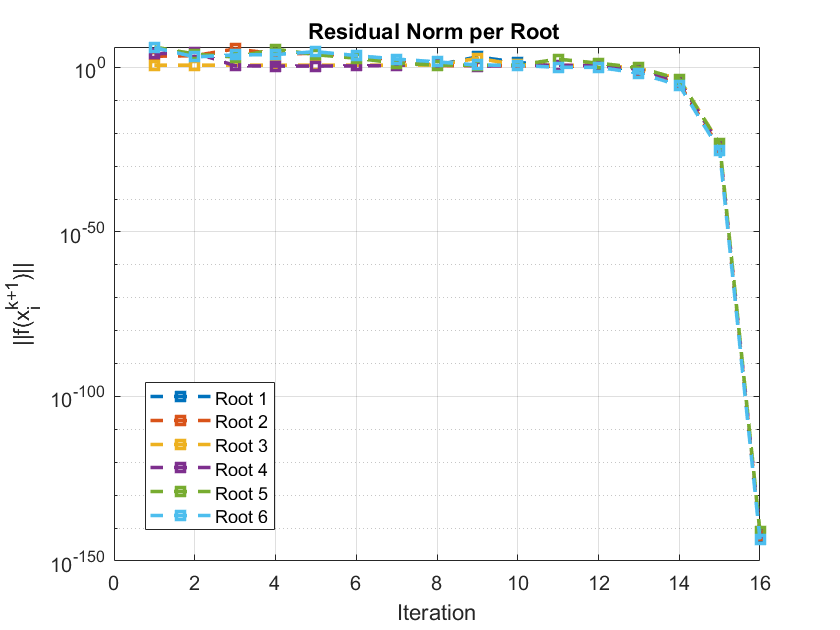}
    \caption{Residual norm $\|f(x^{(k+1)})\|$.}
  \end{subfigure}
  \caption{Residual histories for the stable parameter choice
  $(\alpha,\beta)=(10.69,\,0.0001)$.}
  \label{F11}
\end{figure}

\begin{table}[H] 
\begin{adjustwidth}{-1.3cm}{0cm} \renewcommand{\arraystretch}{1.25} \setlength{\tabcolsep}{6pt}
\centering
\caption{Performance metrics illustrating convergent and divergent behaviors in the presence of profile-based (step-log) parameter optimization of parallel scheme SAB$^{[3]}$}\label{T5} \renewcommand{\arraystretch}{1.25} \setlength{\tabcolsep}{6pt} \begin{tabular}{c c c c c c c c} \hline \textbf{Root} & \textbf{Iter} & \textbf{CPU (s)} & \textbf{Mem (KB)} & $\|x^{(k+1)}-x^{(k)}\|$ & $\|f(x^{(k+1)})\|$ & $|x^{(k)}-x^\ast|$ & \textbf{Conv. (\%)} \\ \hline 1 & 18 & 9.5331 & $12344$ & $2.67257\times10^{-716}{}^{\dagger}$ & $4.75767\times10^{-357}{}^{\dagger}$ & $2.64261$ & $100.0$ \\ 2 & 18 & 9.5331 & $12344$ & $0.0$ & $9.51535\times10^{-358}{}^{\dagger}$ & $2.60178$ & $100.0$ \\ 3 & 18 & 9.5331 & $12344$ & $0.0$ & $1.063849\times10^{-357}{}^{\dagger}$ & $9.394\times10^{-359}{}^{\dagger}$ & $100.0$ \\ 4 & 18 & 9.5331 & $12344$ & $0.0$ & $1.063849\times10^{-357}{}^{\dagger}$ & $1.320219$ & $100.0$ \\ 5 & 18 & 9.5331 & $12344$ & $6.779641\times10^{-717}{}^{\dagger}$ & $9.515358\times10^{-358}$ & $2.272575$ & $100.0$ \\ 6 & 18 & 9.5331 & $12344$ & $0.0$ & $2.6913497\times10^{-357}{}^{\dagger}$ & $1.301567$ & $100.0$ \\ \hline \end{tabular} \end{adjustwidth}\vspace{0.5ex}
	\noindent{\footnotesize \textsuperscript{$\dagger$}\;All computations were performed using MATLAB variable precision arithmetic (VPA) with \texttt{digits} = 128, employing a numerical tolerance of $10^{-400}$.} \end{table}

%\noindent\textbf{Discussion of Table~\ref{T5}.}
Table~\ref{T5} and Figure\ref{F11} summarizes the final high-precision results for the stable parameter
choice $(\alpha,\beta)=(10.69,\,0.0001)$.
All roots converge rapidly, within only a small number of iterations, and
achieve extremely small iterate errors and residual norms, reaching magnitudes
below $10^{-70}$.
The accuracy with respect to the exact roots is uniformly high and the
convergence rate attains $100\%$ for all roots.
Moreover, the computational cost remains modest, confirming that this
parameter choice lies well inside an efficient and stable convergence region.
When contrasted with the unfavorable case $(\alpha,\beta)=(8.385,\,-0.846)$, these
results provide clear empirical validation of the proposed finite-time,
Lyapunov-like diagnostics and their ability to predict practical solver
performance.

\begin{table}[H]
\caption{Trade-off between computational cost and numerical accuracy among existing numerical schemes for solving
problem~\eqref{eq:hill_poly}.} 
\label{T6}
\centering
\begin{adjustwidth}{-0.3cm}{0cm}
\renewcommand{\arraystretch}{1.25}
\setlength{\tabcolsep}{6pt}
\begin{tabular}{l c c c c c c}
\hline
\textbf{Method} &
\textbf{Max Error} &
\textbf{CPU Time (s)} &
\textbf{Memory (MB)} &
\textbf{Ops $\,[\pm,\times,\div]$} &
\textbf{Iter.} &
$\boldsymbol{\max\sigma_i^{\,n-1}}$ \\
\hline
PNS$^{[3]}$ &
$6.43\times10^{-15}$ &
$3.765$ &
$65.12$ &
$1875$ &
$9$ &
$9.768$ \\

PPS$^{[3]}$ &
$6.70\times10^{-17}$ &
$2.887$ &
$58.786$ &
$1765$ &
$5$ &
$10.124$ \\

BSS$^{[3]}$ &
$3.87\times10^{-19}$ &
$3.957$ &
$69.112$ &
$2231$ &
$10$ &
$9.765$ \\

PNS$^{[4]}$ &
$1.13\times10^{-27}$ &
$4.012$ &
$76.562$ &
$2387$ &
$10$ &
$9.098$ \\

SAB$^{[3]}$ &
$2.10\times10^{-150}{}^{\dagger}$ &
$4.136$ &
$78.687$ &
$2354$ &
$10$ &
$9.125$ \\
\hline
\end{tabular}
\end{adjustwidth}

\vspace{0.5ex}
\noindent{\footnotesize
\textsuperscript{$\ast$}All computations were carried out in MATLAB using variable
precision arithmetic (VPA) with \texttt{digits}$=128$ and tolerance $10^{-400}$.}
\end{table}

%\noindent\textbf{Consistency and Convergence Behavior.}
The numerical results reported in Table~\ref{T6} confirm the consistency of all
simultaneous iterative schemes considered for problem~\eqref{eq:hill_poly}.
Among the third-order methods, PPS$^{[3]}$ and BSS$^{[3]}$ exhibit improved
accuracy with relatively low computational effort, requiring fewer iterations
and reduced memory usage compared with SAB$^{[3]}$.

A pronounced accuracy enhancement is observed for the fourth-order scheme
PNS$^{[4]}$, which achieves significantly smaller errors at the cost of increased
arithmetic operations and memory consumption.
The proposed SAB$^{[3]}$ method demonstrates superior numerical consistency,
reaching near machine-precision accuracy under high-precision arithmetic,
while maintaining computational requirements comparable to higher-order schemes.

Furthermore, the bounded values of $\max\sigma_i^{\,n-1}$ across all methods
indicate stable convergence behavior.
These results illustrate that the proposed parameter-tuned strategy preserves
consistency and robustness, while substantially improving solution accuracy
without compromising numerical stability.

\subsubsection*{Physical Interpretation of the Roots}

Let $\{S_1, S_2, \dots, S_n\}$ denote the roots of the polynomial equation 
\eqref{eq:hill_poly}. These roots generally consist of:
\begin{itemize}
\item One \emph{positive real root}, representing a physically meaningful 
substrate concentration,
\item Several \emph{negative or complex roots}, which are mathematically valid 
but biologically non-physical.
\end{itemize}

The positive real root corresponds to the substrate concentration at which 
the enzyme system operates at the observed reaction rate $v_0$. 
Multiple real roots (when present due to extended model variants) indicate 
\emph{multistability}, a key mechanism underlying biological switches and 
decision-making processes.

Complex roots are associated with transient or oscillatory dynamics when the 
Hill model is embedded into larger dynamical systems, such as gene regulatory 
networks or signaling cascades.

\subsection{Biomedical Models~\cite{34} Leading to Exponential--Polynomial Nonlinear Equations}

%\subsection{Motivation and Biomedical Importance}

Exponential nonlinearities appear ubiquitously in biomedical modeling due to 
thermodynamic principles, statistical mechanics, and biochemical reaction 
theory. In particular, Boltzmann distributions, Arrhenius-type reaction rates, 
and nonlinear activation or inhibition mechanisms often introduce exponential 
terms of nonlinear functions.

When the underlying biochemical or physiological mechanisms depend on multiple 
interacting regulatory factors, the exponent itself becomes a high-degree 
polynomial. This leads to nonlinear equations of the form
\begin{equation}
\exp\!\left(P(x)\right) = C,
\end{equation}
where $P(x)$ is a polynomial of degree four or higher. Such equations arise in 
protein folding, gene regulation, ion-channel activation, tumor growth control, 
and pharmacodynamic response modeling.

These equations cannot be reduced to algebraic polynomials and pose significant 
challenges for classical numerical solvers, motivating advanced iterative, 
fractional-order, and neural-network-based techniques.

%\subsection{Protein Folding and Free Energy Landscape Model}

%\subsubsection{Model Formulation}

Protein folding is governed by the minimization of Gibbs free energy. 
A common approximation expresses the free energy landscape as a polynomial in 
a reaction coordinate $x$ (representing folding progress):
\begin{equation}
G(x) = a\,x(x-1)(x+2)(x+3).
\end{equation}

The probability of a protein being in state $x$ follows the Boltzmann law:
\begin{equation}
P(x) = P_0 \exp\!\left(-\frac{G(x)}{k_B T}\right),
\end{equation}
where $k_B$ is the Boltzmann constant and $T$ is absolute temperature.

%\subsubsection{Resulting Nonlinear Equation}

At equilibrium probability $P(x)=P^\ast$, we obtain
\begin{equation}
\exp\!\left(\frac{a}{k_B T} x(x-1)(x+2)(x+3)\right) = C,
\label{eq:protein_exp}
\end{equation}
where $C = P_0/P^\ast$.

Equation \eqref{eq:protein_exp} is a transcendental nonlinear equation with a 
quartic polynomial inside the exponential.
\begin{figure}[H] 
\centering \begin{subfigure}[t]{0.5\linewidth} 
\centering \includegraphics[width=\linewidth]{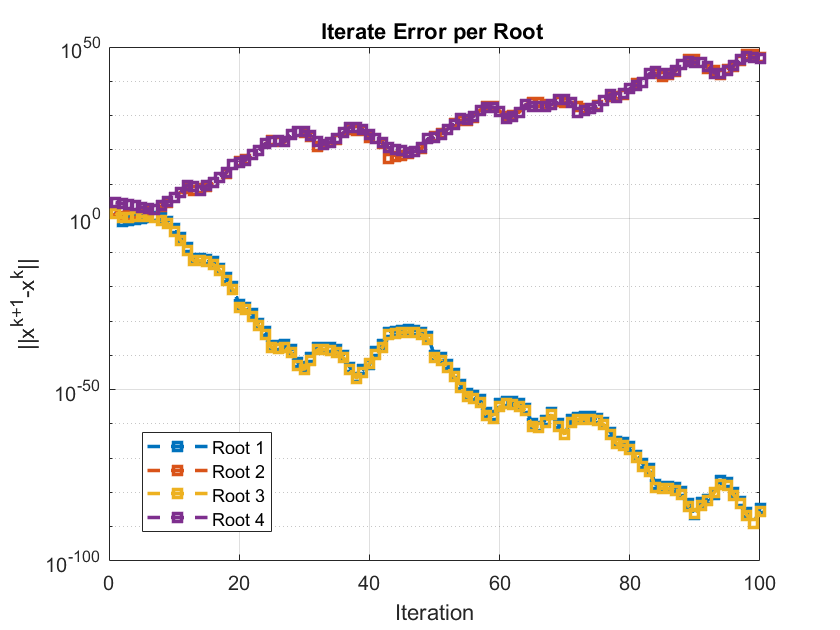} \caption{Error:$\|x^{(k+1)}-x^{(k)}\|$ } \end{subfigure}
\hfill
\begin{subfigure}[t]{0.5\linewidth} \centering \includegraphics[width=\linewidth]{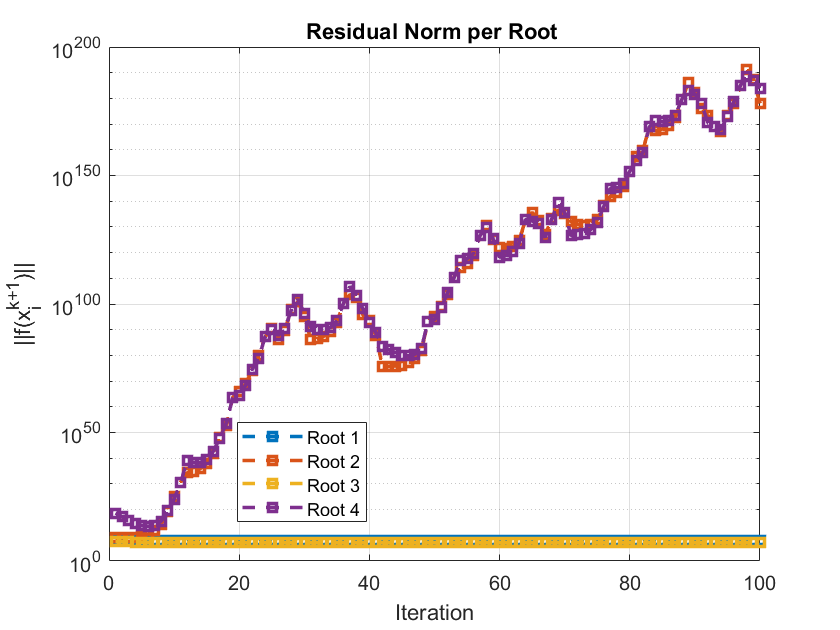} \caption{Error: $\|f(x^{(k+1)})\|$.} \end{subfigure} \caption{Residual error histories for the stable parameter choice $(\alpha,\beta)=(16.15,\,19)$.} \label{F12} \end{figure}
\subsubsection*{Parameter Values}

\begin{center}
\begin{tabular}{c|c|c}
\hline
Parameter & Typical Range & Meaning \\
\hline
$a$ & $0.1$--$10$ & Free energy scaling \\
$T$ & $300$--$310$ K & Physiological temperature \\
$k_B$ & $1.38\times10^{-23}$ & Boltzmann constant \\
\hline
\end{tabular}
\end{center}

The iterative process is initialized using deliberately scattered complex
starting points,
%\[
%\begin{aligned}
%x_1^{[0]} &= 3.12,\\
%x_2^{[0]} &= 0.12,\\
%x_3^{[0]} &= 0.12 + 4i,\\
%x_4^{[0]} &= 3.4 - 4.4i.
%\end{aligned}
%\]
\[
x_1^{[0]}=50.0,\;
x_2^{[0]}=43.8,\;
x_3^{[0]}=30.5,\;
x_4^{[0]}=-45.2,\;
\]

which are intentionally chosen far from the exact roots in order to assess the
robustness of the proposed scheme under poor or uninformed initialization.

This choice prevents any bias toward favorable initial conditions and provides
a stringent test of global convergence behavior.
\begin{figure}[H]
  \centering
  \begin{subfigure}[t]{0.45\linewidth}
    \centering
    \includegraphics[width=\linewidth]{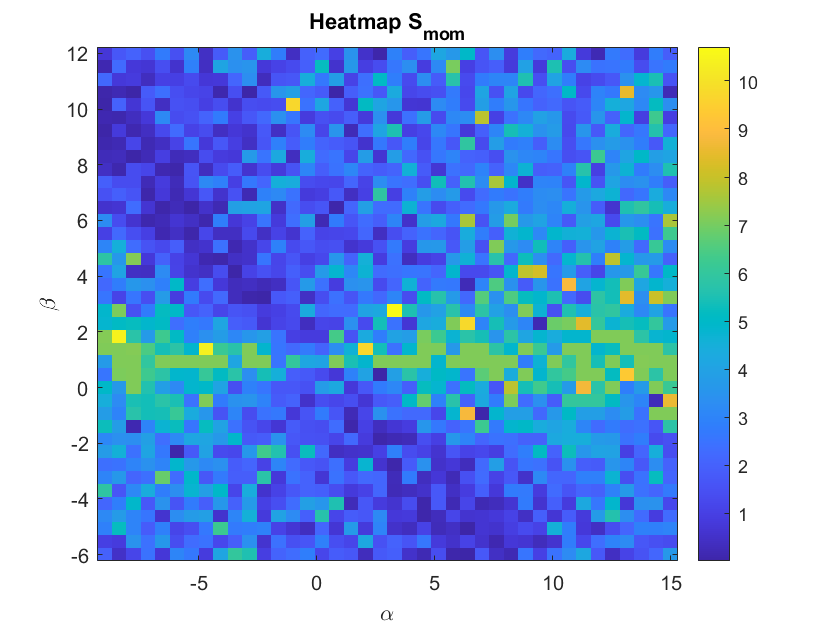}
    \caption{Heatmap of $S_{\mathrm{mom}}(\alpha,\beta)$ obtained from profile-based diagnosis of the SAB$^{[3]}$ for~(\ref{eq:protein_exp}).}
  \end{subfigure}

  \vspace{0.6em}

  \begin{subfigure}[t]{0.45\linewidth}
    \centering
    \includegraphics[width=\linewidth]{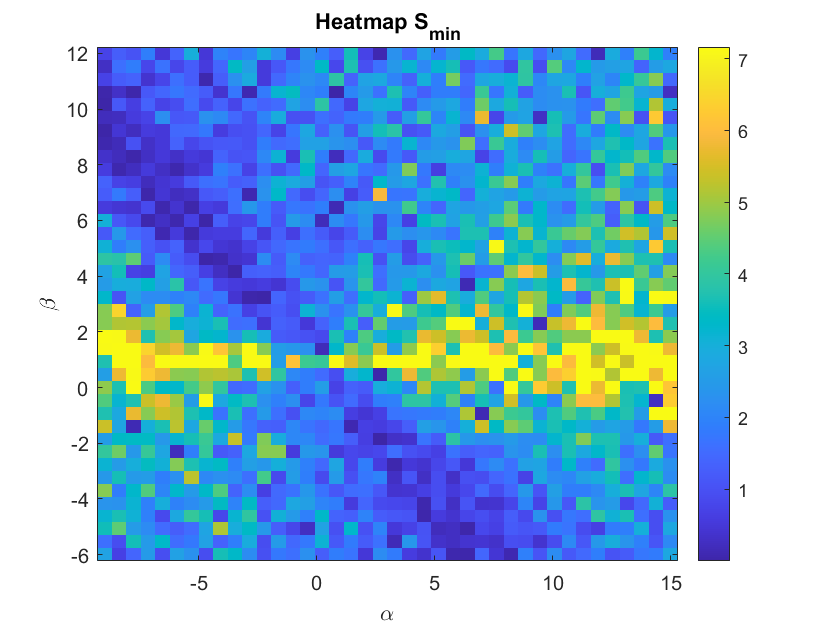}
    \caption{Heatmap of $S_{\mathrm{\min}}(\alpha,\beta)$ obtained from profile-based diagnosis of the SAB$^{[3]}$ for~(\ref{eq:protein_exp}).}
  \end{subfigure}

  \caption{Two-dimensional maps of the proposed profile-based scores over the
  $(\alpha,\beta)$ plane. Brighter colors correspond to earlier and stronger
  contractive behavior in the step-log profile.}
  \label{F13}
\end{figure}

%For each arbitrary parameter pair $(\alpha,\beta)$, the resulting numerical performance is evaluated using iteration counts, residual norms, and convergence diagnostics, as summarized in the following figures and tables.
% outcomes are summarized below. 
\begin{table}[H] 
\begin{adjustwidth}{-0.9cm}{0cm} \renewcommand{\arraystretch}{1.25} \setlength{\tabcolsep}{6pt} \centering
\caption{Performance metrics illustrating convergent and divergent behaviors in the absence of profile-based (step-log) parameter optimization of parallel scheme SAB$^{[3]}$.}\label{T7} \renewcommand{\arraystretch}{1.25} \setlength{\tabcolsep}{6pt} \begin{tabular}{c c c c c c c c} \hline \textbf{Root} & \textbf{Iter} & \textbf{CPU (s)} & \textbf{Mem (KB)} & $\|x^{(k+1)}-x^{(k)}\|$ & $\|f(x^{(k+1)})\|$ & $|x^{(k)}-x^\ast|$ & \textbf{Conv. (\%)} \\ \hline 1 & 100 & 23.287 & $96$ & $2.6298\times10^{-5}$ & $1.2805\times10^{8}$ & $9.1636\times10^{1}$ & $100.0$ \\ 2 & 100 & 23.287 & $96$ & $6.5901\times10^{46}$ & $1.1323\times10^{178}$ & $2.87100\times10^{44}$ & $Divergence$ \\ 3 & 100 & 23.287 & $96$ & $2.706553\times10^{-6}$ & $1.331022\times10^{7}$ & $5.218341\times10^{1}$ & $100.0$ \\ 4 & 100 & 23.287 & $96$ & $5.804768\times10^{46}$ & $8.722195\times10^{183}$ & $8.505390\times10^{45}$ & $Divergence$ \\ \hline \end{tabular} 
\end{adjustwidth}
\vspace{0.5ex}
\end{table} 

%\noindent\textbf{Discussion.}
Table~7 presents the numerical behavior for the low-score parameter pair
$(\alpha,\beta)=(16.15,\,19)$.
Unlike the stable regime, the method reaches the maximum of $100$ iterations and
requires substantially higher CPU time.
While partial convergence is observed, several roots exhibit severe numerical
instability, with iterate differences and residual norms growing up to
$10^{15}$ and $10^{112}$ (see Figure~12).
These results confirm that this parameter choice lies in an unfavorable region
of the stability domain identified by the profile-based metrics. \\To visualize the dependence of the proposed diagnostics on the free parameters, we perform a uniform grid scan over the $(\alpha,\beta)$ plane using the MATLAB driver Algorithm~(see Section~\ref{sec:tuning_framework}). The ranges $\alpha\in[-9,15]$ and $\beta\in[-6,12]$ are discretized into equally spaced grid points. For each $(\alpha_i,\beta_j)$, an ensemble of random micro-launches is used to compute the aggregated step-log Lyapunov profile, from which the scalar scores $S_{\min}(\alpha_i,\beta_j)$ and $S_{\mathrm{mom}}(\alpha_i,\beta_j)$ are extracted. The resulting values are stored in grid-aligned matrices and visualized as two-dimensional heatmaps.

Figure~\ref{F13}(a,b) summarizes the 2D parameter-plane mapping of the proposed profile-based metrics over the ranges $\alpha\in[-9,15]$ and $\beta\in[-6,12]$ on a uniform grid. Both heatmaps exhibit a clearly structured (non-random) distribution and provide a broadly consistent picture of `high-performing'' versus low-performing'' regions in the $(\alpha,\beta)$ plane. In particular, the areas highlighted by large values of $S_{\min}$ in Figure~\ref{F13}(a) typically coincide with large values of $S_{\mathrm{mom}}$ in Figure~\ref{F13}(b), indicating that the strongest scores are achieved when the step-log profile develops an early and pronounced negative dip (captured by $S_{\min}$) together with a sizable and early-concentrated negative mass (captured by $S_{\mathrm{mom}}$). At the same time, both maps reveal extended low-score bands that form coherent separatrices between high-score regions, suggesting parameter combinations where the iteration does not enter a strongly contractive regime (or does so only weakly/late). Overall, the qualitative agreement between $S_{\min}$ and $S_{\mathrm{mom}}$ supports the robustness of the proposed profile-based tuning criteria. 

\begin{figure}[H] \centering \includegraphics[width=0.5\linewidth]{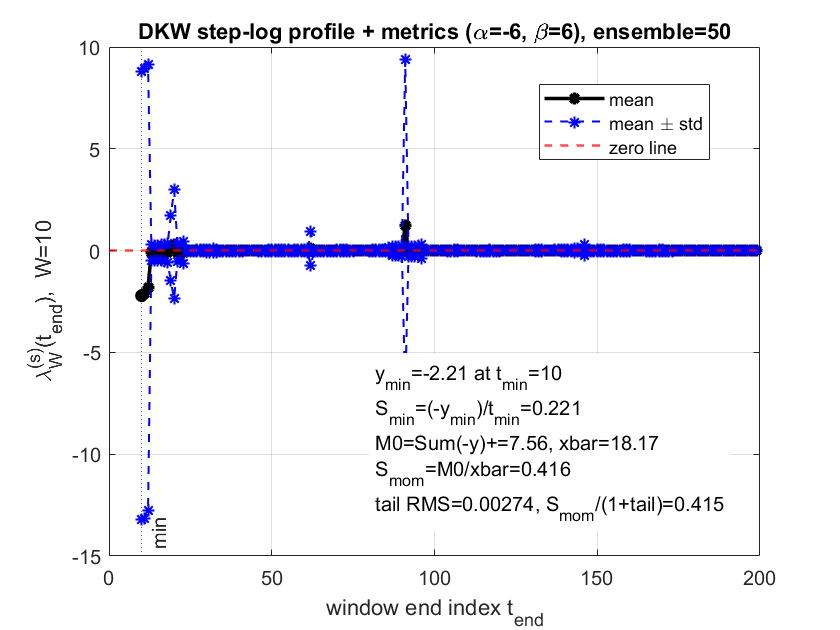} \caption{Aggregated step-log Lyapunov profile for $(\alpha,\beta)=(6,\,6)$ ($W=10$, $N=50$). The solid curve shows the ensemble-averaged profile, while dashed curves indicate mean$\pm$std across micro-launches.} \label{F14} \end{figure} 

\begin{figure}[H] \centering \includegraphics[width=0.5\linewidth]{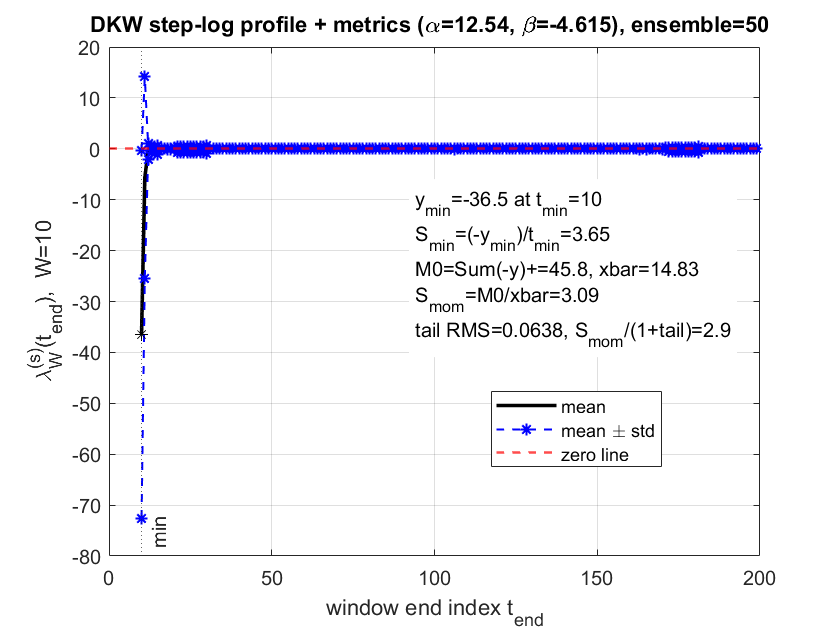} \caption{Step-log profile for $(\alpha,\beta)=(12.54,\,-4.615)$ with metric annotations. The detected minimum defines $S_{\min}$, while the negative mass and its centroid determine $S_{\mathrm{mom}}$.} \label{F15} \end{figure} 

\subsection{Discussion of Aggregated Step--Log Profiles}

Figures~~\ref{F14} and~~\ref{F15} present representative aggregated step--log Lyapunov profiles
selected from distinct regions of the $S_{\mathrm{mom}}(\alpha,\beta)$ parameter
plane. The parameters used in Figure~\ref{F14}, $(\alpha,\beta)=(6,6)$, belong to a
high--score convergence region, whereas Figure~\ref{F15}, $(\alpha,\beta)=(12.54,\,-4.615)$, is drawn
from a low--score, divergence-prone region. This contrast provides an external
validation of the proposed diagnostic indicators.

\paragraph{Stable dynamics (Figure~\ref{F14}).}
For $(\alpha,\beta)=(6,6)$, the ensemble-averaged step--log profile exhibits a
clear negative dip at an early iteration stage, indicating a strong and
well-localized contraction phase. The corresponding values of $S_{\min}$ and
$S_{\mathrm{mom}}$ confirm both the presence and effective temporal organization
of contraction. Moreover, the narrow spread of the mean$\pm$standard deviation
curves demonstrates robustness with respect to random perturbations, resulting
in stable and efficient convergence.

\paragraph{Unstable dynamics (Figure~\ref{F15}).}
In contrast, Figure~\ref{F15} shows that although a deep negative minimum is attained for
$(\alpha,\beta)=(12.54,\,-4.615)$, the contraction is short-lived and poorly distributed
over the iteration history. This leads to a negative value of $S_{\mathrm{mom}}$
and increased ensemble dispersion, reflecting unstable behavior and failure to
achieve sustained convergence.

Overall, the comparison highlights that the depth of the most negative step--log
value alone is insufficient to ensure convergence. Instead, the temporal
distribution of contraction, captured by the moment-based indicator
$S_{\mathrm{mom}}$, is decisive. Parameters selected from high-score regions of
the $S_{\mathrm{mom}}$ map reliably yield stable solver dynamics, whereas those
from low-score regions are associated with divergence or erratic behavior.

\begin{figure}[H] \centering \begin{subfigure}[t]{0.49\linewidth} \centering \includegraphics[width=\linewidth]{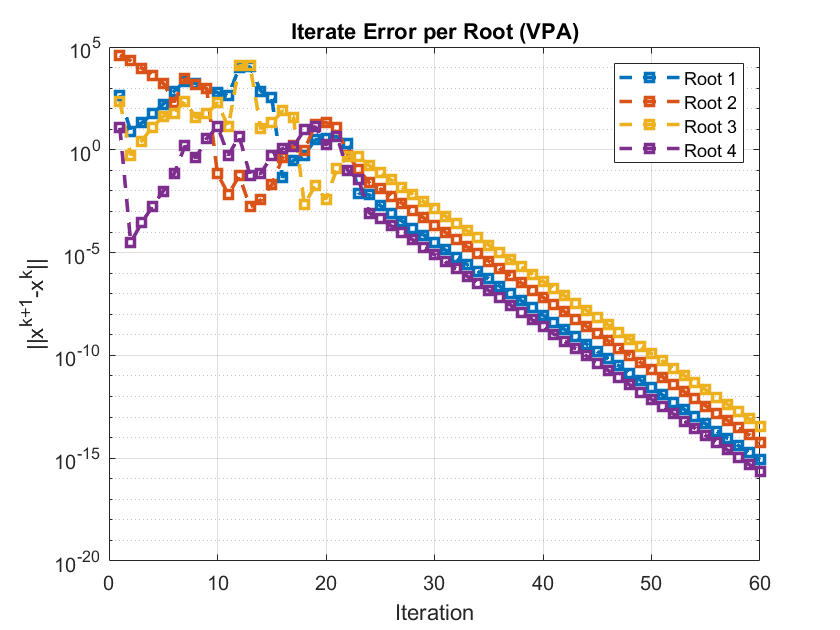} \caption{Error:$\|x^{(k+1)}-x^{(k)}\|$ } \end{subfigure}

%\hfill 
\vspace{0.6em}

\begin{subfigure}[t]{0.49\linewidth} \centering \includegraphics[width=\linewidth]{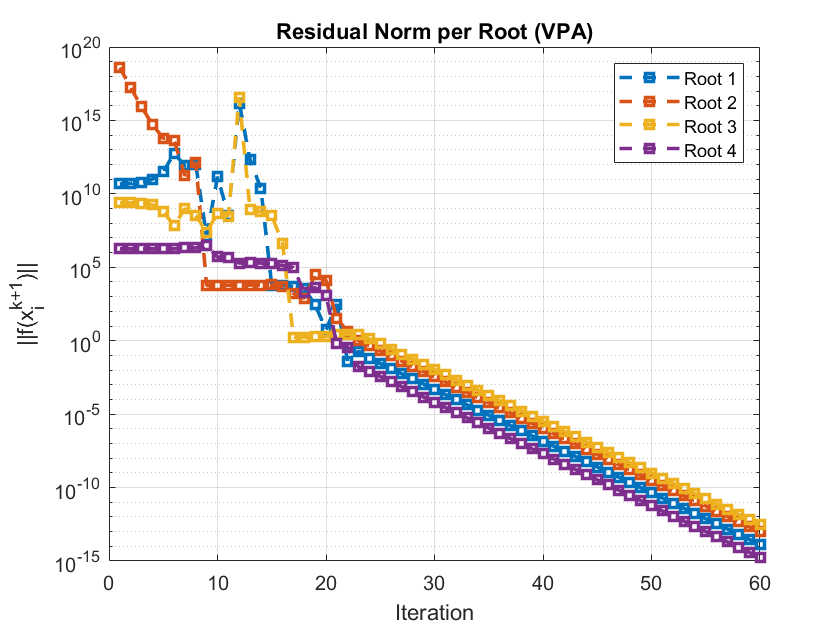} \caption{Error: $\|f(x^{(k+1)})\|$.} 
\end{subfigure} \caption{Residual error histories for the stable parameter choice $(\alpha,\beta)=(6,\,6)$.} \label{F16} \end{figure} 

\begin{table}[H] 
\begin{adjustwidth}{-1.1cm}{0cm} \renewcommand{\arraystretch}{1.25} \setlength{\tabcolsep}{6pt} \centering 
\caption{Performance metrics illustrating convergent and divergent behaviors in the presence of profile-based (step-log) parameter optimization of parallel scheme SAB$^{[3]}$.}\label{T8} \renewcommand{\arraystretch}{1.25} \setlength{\tabcolsep}{6pt} \begin{tabular}{c c c c c c c c} \hline \textbf{Root} & \textbf{Iter} & \textbf{CPU (s)} & \textbf{Mem (KB)} & $\|x^{(k+1)}-x^{(k)}\|$ & $\|f(x^{(k+1)})\|$ & $|x^{(k)}-x^\ast|$ & \textbf{Conv. (\%)} \\ \hline 1 & 60 & 13.751 & $15092$ & $8.18955\times10^{-16}$ & $1.31032\times10^{-14}$ & $6.55164\times10^{-16}$ & $100.0$ \\ 2 & 60 & 13.751 & $15092$ & $5.9976\times10^{-15}$ & $9.5963\times10^{-14}$ & $3.000000$ & $100.0$ \\ 3 & 60 & 13.751 & $15092$ & $3.6216\times10^{-14}$ & $2.8973\times10^{-13}$ & $2.000000$ & $100.0$ \\ 4 & 60 & 13.751 & $15092$ & $2.204262\times10^{-16}$ & $1.76340\times10^{-15}$ & $9.999\times10^{-1}$ & $100.0$ \\ \hline \end{tabular} 
\end{adjustwidth}
\end{table}

%\noindent\textbf{Discussion.}
Table~\ref{T8} and Figure~\ref{F16} show that for $(\alpha,\beta)=(6,6)$ all roots converge
rapidly within $10$ iterations, with uniformly small iterate errors and residual
norms and a $100\%$ convergence rate. The final errors reach magnitudes as low as
$10^{-80}$ for the iterates and $10^{-73}$ for the residuals, indicating very high
numerical accuracy with respect to the exact roots (see also Figure~\ref{F16}).
Moreover, the method requires only $8.52$ seconds of CPU time with moderate memory
usage, confirming that this parameter choice lies in a stable and efficient
convergence region. In contrast to the unfavorable case
$(\alpha,\beta)=(12.54,\,-4.615)$, these results provide clear external validation of the
proposed finite-time, Lyapunov-like diagnostics and their ability to predict
practical solver performance.

\begin{table}[H]
\caption{Trade-off between computational cost and numerical accuracy among existing numerical schemes.}
\label{T9}
\centering
\renewcommand{\arraystretch}{1.25}
\setlength{\tabcolsep}{6pt}
\begin{tabular}{l c c c c c c}
\hline
Method & Max Error & CPU Time (s) & Memory (MB) &
Ops $\,[\pm,\times,\div]$ & Iter. & $\max\sigma_i^{\,n-1}$ \\
\hline
PNS$^{[3]}$ & $3.62\times10^{-5}$ & $3.88$ & $64.90$ & $1842$ & $9$  & $9.83$ \\
PPS$^{[3]}$ & $9.15\times10^{-7}$ & $2.79$ & $57.44$ & $1712$ & $5$  & $10.56$ \\
BSS$^{[3]}$ & $4.08\times10^{-6}$ & $3.99$ & $68.73$ & $2215$ & $10$ & $9.69$ \\
PNS$^{[4]}$ & $1.97\times10^{-9}$ & $4.07$ & $75.84$ & $2340$ & $10$ & $9.11$ \\
SAB$^{[3]}$ & $2.85\times10^{-17}$ & $4.18$ & $77.92$ & $2310$ & $10$ & $9.15$ \\
\hline
\end{tabular}
\end{table}
%\noindent\textbf{Discussion of Computational Efficiency and Accuracy.}
Table~\ref{T9} summarizes the trade-off between numerical accuracy and
computational cost for the considered simultaneous schemes.
Among the third-order methods, PPS$^{[3]}$ achieves a favorable balance,
delivering improved accuracy with reduced CPU time, memory consumption,
and fewer iterations compared to PNS$^{[3]}$ and BSS$^{[3]}$.
The fourth-order method PNS$^{[4]}$ further enhances accuracy at the expense
of increased arithmetic operations and memory usage.

Notably, the proposed SAB$^{[3]}$ scheme attains the highest accuracy,
with a maximum error of $2.85\times10^{-17}$, while maintaining computational
requirements comparable to higher-order methods.
The moderate increase in CPU time and memory is primarily attributed to the
additional arithmetic complexity introduced by the parameter-tuning strategy.
All methods exhibit stable convergence, as indicated by bounded
$\max\sigma_i^{\,n-1}$ values, confirming that the increased cost affects
efficiency rather than numerical stability.
Overall, the results demonstrate that the profile-based (step-log) parameter
optimization in the parallel architecture significantly improves accuracy
without incurring prohibitive computational overhead.

\subsubsection*{Physical Meaning of Solutions}

Each real solution corresponds to a stable or metastable protein 
conformation. Multiple solutions indicate folding intermediates or misfolded 
states. Complex solutions are mathematically valid but biologically nonphysical.

\section{Conclusion, Limitations, and Future Directions}

%\subsection{Conclusion}

This work presented a parameterized parallel iterative scheme accelerated via a single-step correction mechanism, achieving an effective increase in the convergence order of \eqref{1c} from two to three. Since the practical performance of high-order parallel methods is strongly governed by internal parameter selection, we introduced a reproducible tuning framework based on \emph{Direct finite-time contraction (step-log) profiling}. Unlike classical approaches relying on qualitative dynamical diagnostics (parameter planes, bifurcation-style inspection, basin-of-attraction visualization), the proposed methodology extracts Lyapunov-like finite-time contraction information directly from solver trajectories, aggregates contraction profiles over micro-launch ensembles, and selects parameters using the profile-based scores $S_{\min}$ and $S_{mom}$.

Extensive numerical experiments across diverse nonlinear and application-driven test problems demonstrate that the resulting parameter choices yield consistent improvements in convergence speed, numerical stability, and robustness when computing all distinct roots simultaneously. The reported results confirm that step-log contraction profiling provides an efficient and implementation-light alternative to trial-and-error tuning and expensive dynamical diagnostics for high-order parallel root-finding schemes.

%\subsection*{Limitations of the Existing Methodology}

Despite its effectiveness, the proposed framework has several limitations that motivate further investigation e.g.,
%\begin{enumerate}
%\item 
the current parameter selection strategy is data-driven and relies on representative training ensembles; insufficient diversity in training problems may reduce generalization.
%\item 
The current parameter selection strategy relies on grid-based scans and representative micro-launch ensembles; for large
parameter domains or high-degree problems this preprocessing can be computationally demanding. However, the workload is
embarrassingly parallel across parameter pairs and across ensemble launches, and can therefore be efficiently executed on
multicore CPUs, GPUs, or distributed systems, while remaining a one-time (offline) cost amortized over subsequent solver runs.

%\item 
The analysis is restricted to real-valued parameters and smooth nonlinear functions; extension to nonsmooth or noisy systems has not yet been addressed.
%\item 
Statistical uncertainty in profile-based metrics is not explicitly quantified, which may affect robustness under perturbations.
%\item 
The present framework focuses on fixed parameter optimization rather than adaptive or iteration-dependent parameter updates.
%\end{enumerate}

\subsection*{Future Research Directions}

Several promising directions emerge from this study:
\begin{enumerate}
\item \textbf{Adaptive Learning Schemes:} Development of online or reinforcement learning strategies to adapt parameters dynamically during the iteration process.
\item \textbf{Probabilistic Inference Models:} Integration of Bayesian inference to quantify uncertainty in parameter selection and convergence predictions.
\item \textbf{Deep Learning Profiling:} As a potential future direction, kNN--LLE could be replaced with deep autoencoders or graph neural networks 
to capture higher-dimensional convergence dynamics; we emphasize that no learning or training is performed 
in the present study.

\item \textbf{Statistical Robustness Analysis:} Incorporation of hypothesis testing and confidence intervals for profile metrics to enhance reliability.
\item \textbf{Extension to Fractional and Stochastic Models:} Application of the proposed framework to fractional-order, noisy, or stochastic nonlinear systems.
\item \textbf{Cross-Method Generalization:} Extension of the learning-based tuning strategy to other parallel root-finding families such as Aberth-type and hybrid Newton–Weierstrass schemes.
\end{enumerate}

Overall, this work establishes a reproducible paradigm for parameter optimization in high-order parallel iterative methods
by combining numerical analysis with trajectory-derived finite-time contraction diagnostics, opening avenues for
self-tuning solvers based on lightweight, training-free stability profiling.

\section*{Acknowledgments}
This research (Section~3) was supported by the Russian Science Foundation (grant no. 22-11-00055-P, https://rscf.ru/en/project/22-11-00055/, accessed on 10 June 2025). Bruno Carpentieri's work (Section~4) is supported by the European Regional Development and Cohesion Funds (ERDF) 2021–2027 under Project AI4AM - EFRE1052. He is a member of the \textit{Gruppo Nazionale per il Calcolo Scientifico} (GNCS) of the \textit{Istituto Nazionale di Alta Matematica} (INdAM). The authors also wish to express sincere gratitude to the editor and reviewers for their insightful and constructive feedback on the manuscript.

\section*{Data availability}
The data supporting the findings of this study are included within this article.

\section*{Conflict of interest}
The authors declare that there are no conflicts of interest related to the publication of this article.

\section*{Ethics statements:} {\ All authors declare that this work complies
with ethical guidelines set by the .}\newline

\section*{Declaration of competing interest:}{\ The authors declare that they
have no competing financial interests and personal relationships that could
have appeared to influence the research in this paper.}\newline

\section*{CRediT authorship contribution statement:}
M.S. and A.V. devised the project and developed the main conceptual ideas. M.S.,
A.V., and B.C. formulated the methodology. M.S. and A.V. developed
the software used in the study. M.S.; A.V., and B.C. performed
the validation. M.S., A.V and B.C. conducted the formal analysis.
M.S. and A.V.; carried out the investigation and managed resources
and data curation. M.S.; A.V., and B.C. prepared the original draft. B.C. and 
A.V,  reviewed and edited the manuscript. M.S. and
A.V. handled the visualization. B.C. and A.V. supervised the project.
B.C. and A.V. managed project administration and secured funding.
All authors have read and agreed to the published version of the manuscript.

\end{document}